\documentclass{article}
\usepackage{amsfonts}
\usepackage{amssymb}
\usepackage{amsmath}
\usepackage{lineno}
\usepackage{graphicx}
\usepackage{color}

\newtheorem{theorem}{Theorem}

\newtheorem{lemma}[theorem]{Lemma}

\newtheorem{remark}[theorem]{Remark}

\newenvironment{proof}[1][Proof]{\noindent\textbf{#1.} }{\ \rule{0.5em}{0.5em}}

\begin{document}

\title{New error estimates of Lagrange-Galerkin methods for the advection
equation}
\author{Rodolfo Bermejo$^{a}$, Jaime Carpio$^{b}$, Laura Saavedra$^{c}$ \\
{\small $a$) Dpto. Matem\'{a}tica Aplicada a la Ingenier\'{\i}a Industral
ETSII. Universidad Polit\'ecnica de Madrid. }\\
[0pt] {\small $b$) Dpto. Ingenier\'ia Energ\'etica ETSII. Universidad
Polit\'ecnica de Madrid. }\\
[0pt] {\small $c$) Dpto. Matem\'{a}tica Aplicada a la Ingenier\'{\i}a
Aeroespacial, ETSIAE. Univversidad Polit\'ecnica de Madrid.}}
\date{}
\maketitle

\begin{abstract}
We study in this paper new developments of the Lagrange-Galerkin method for
the advection equation. In the first part of the article we present a new
improved error estimate of the conventional Lagrange-Galerkin method. In the
second part, we introduce a new local projection stabilized
Lagrange-Galerkin method, whereas in the third part we introduce and analyze
a discontinuity-capturing Lagrange-Galerkin method. Also, attention has been
paid to the influence of the quadrature rules on the stability and accuracy
of the methods via numerical experiments.
\end{abstract}


%


discontinuity capturing 


\section{Introduction}

We consider the Cauchy problem for the pure advection equation%
\begin{equation}
\left\{
\begin{array}{l}
\dfrac{\partial c}{\partial t}+\mathbf{u\cdot }\nabla c=0,\ \ x\in \mathbb{R}%
^{d},\ t>0,\  \\
\\
c(x,0)=c^{0}(x),%
\end{array}%
\right.  \label{eq0}
\end{equation}%
where $c:\mathbb{R}^{d}\times \lbrack 0,T]\rightarrow \mathbb{R}$, $\mathbf{u%
}:\mathbb{R}^{d}\times \lbrack 0,T]\rightarrow \mathbb{R}^{d}$ is a
vector-valued function and $c^{0}(x)$ is a function of compact support
defined in a domain $D_{0}\subset \mathbb{R}^{d}$. It is well known that the
solution of this problem given by the method of characteristics is of the
form%
\begin{equation*}
c(X\mathbf{(}\cdot \mathbf{,\ }t;t+\tau ),t+\tau )=c(\cdot ,t),
\end{equation*}%
$X\mathbf{(}x\mathbf{,\ }t;t+\tau )$ being the characteristic curves of the
advection equation. In this paper, we present the error analysis of three
versions of the so called LG method applied to solve the Cauchy problem (\ref%
{eq0}), which represent numerical realizations of the theoretical method of
characteristics in the framework of $H^{1}$-conforming finite elements. The
first version, denoted in this paper with the name the conventional LG
method, consists basically on approximating the solution $c(x,t)$ by the $%
L^{2}$-projection onto the finite element space; see, for instance, \cite{Pi}%
, \cite{MPS} and \cite{john}. The conventional LG method can be viewed as a
kind of high order upwind method that introduces artificial diffusion in the
discrete formulation, thus providing good stability properties to the
numerical solution; however, as numerical experiments show, this artificial
diffusion is not high enough to suppress the oscillations that appear at
discontinuities of the exact solution. In order to alleviate this problem at
discontinuities and extend the stability properties, we shall study a local
projection stabilized Lagrange-Galekin (LPS-LG) method and a
discontinuity-capturing Lagrange-Galerkin (DC-LG) method.

In the past, several authors have obtained different estimates for the $%
L^{2} $-norm of the error of the conventional LG method. For example, \cite%
{Pi} calculates an estimate of the form $O(h^{m+1}/\Delta t)$, where $m$
denotes the degree of the polynomials of the $H^{1}$-conforming finite
element spaces, $h$ being the largest diameter of the elements of the
spatial mesh and $\Delta t$ the size of the time step. The problem with this
estimate is that for fixed $h$ the error becomes unbounded when $\Delta
t\rightarrow 0$. \cite{MPS} removes the $\Delta t^{-1}$ dependence from the
error estimate obtaining the new estimate $O(h^{m})$, this estimate allows
to prove convergence of LG method for the advection equation when $\Delta
t\rightarrow 0$ independently of $h$. \cite{john} improves the estimate of
\cite{Pi} calculating a new estimate $O(h^{m+1}/\Delta t^{1/2})$, which for $%
\Delta t=O(h)$ implies that the error of the conventional LG method is of
the same order as both the streamline-diffusion (SD) method formulated in
the framework of space-time finite elements continuous in space and
discontinuous in time and the characteristic streamline-diffusion (CSD)
method, the latter method being a version of the SD method that uses
space-time meshes oriented along the characteristic curves of the advection
equation. Later on, \cite{MS} calculates a new estimate of the form $O(\min
(1,\left\Vert \mathbf{u}\right\Vert _{L^{\infty }(\mathbb{R}^{d})^{d}}{%
\Delta t}/{h})h^{m+1}/{\Delta t})$, where $\left\Vert \mathbf{u}\right\Vert
_{L^{\infty }((0,T);\mathbb{R}^{d})^{d}}$ denotes the supremum norm of the
velocity vector $\mathbf{u}(x,t)$; then, considering that $\left\Vert
\mathbf{u}\right\Vert _{L^{\infty }(\mathbb{R}^{d})^{d}}{\Delta t}/{h}$ is
the CFL number, we can say that for CFL numbers less than one, the error of
LG methods is $O(h^{m})$, whereas for CFL numbers larger than one the error
is $O(h^{m+1}/{\Delta t})$. Numerical examples show that the latter estimate
provides a better description of the error behavior of the conventional LG
method than the other estimates do. In this paper, we revisit the results of
the the above mentioned authors and calculate an improved new error estimate
of the form $O(\min (1,\left\Vert \mathbf{u}\right\Vert _{L^{\infty }(%
\mathbb{R}^{d})^{d}}{\Delta t}^{1/2}/{h})h^{m+1}/{\Delta t}^{1/2})$. Some
numerical examples will support the validity of this estimate.

Local projection stabilized methods have become quite popular for
advection-diffusion-reaction equations, including Navier-Stokes equations,
see \cite{BB}, \cite{BCS}, \cite{BS3}, \cite{GT} and \cite{RST} just to cite
a few, for they are symmetric and introduce artificial diffusivity via a
fluctuation operator acting on the small unresolved scales. We prove that
the LPS-LG method is stable in the $L^{2}$-norm, and our error analysis
shows that the error of the LPS-LG method in the mesh dependent norm (to be
defined below)%
\begin{equation*}
\max_{0\leq t_{n}\leq T}\left\vert \left\vert \left\vert
c^{n}-c_{h}^{n}\right\vert \right\vert \right\vert =O(h^{\beta }+\min
(1,\left\Vert \mathbf{u}\right\Vert _{L^{\infty }(\mathbb{R}^{d})^{d}}{%
\Delta t}^{1/2}/{h})h^{m+1}/{\Delta t}^{1/2}),
\end{equation*}%
where $\beta $ is a coefficient depending on $m$. However, despite the
introduction of the artificial diffusivity, numerical tests show that the
LPS-LG method may exhibit an oscillatory behavior when the solution is not
sufficiently smooth.

The DC-LG method might be viewed as a version of the shock capturing CSD
method \cite{john} in which the mesh alignment along the characteristic
curves and the stream diffusion mechanism of the shock capturing CSD are
removed; in fact, one can consider that the DC-LG method is a reformulation,
in the framework of the conventional LG method, of the residual artificial
viscosity method introduced in \cite{Naz}. We prove that DC-LG method is
stable in the $L^{2}$-norm, regardless the degree of the finite element
spaces, and also in the $L^{\infty }$-norm with linear finite elements,
although numerical examples show $L^{\infty }$-norm stability with quadratic
polynomials in the presence of a strong discontinuity. For solutions
sufficiently smooth, we are able to prove that the error in the $L^{2}$-norm
is of the form $O((h^{m+1}+C_{\varepsilon }h^{\alpha })/(\Delta t^{1/2}))$, $%
\alpha $ and $C_{\varepsilon }$ being positive constants.

The theoretical analysis of LG methods presented in this paper are proven
under the assumption that the integrals $\int_{K}\phi
_{j}(X(x,t_{n};t_{n-1}))\phi _{i}(x)dx$ , which appear in the formulation of
the methods, are calculated exactly; here, $K$ is a generic element of the
mesh, $\phi _{i}$ is the ith global basis function of the finite element
space and $X(x,t_{n},t_{n-1})$ is the foot of the characteristic curve
associated with the point $x$. Noting that the integrand is the product of
two piecewise continuous polynomial functions defined on two different
meshes, it may become very difficult to calculate such integrals exactly, so
one has to resort to quadrature rules; but as \cite{MPS} and \cite{jack}
show, the quadrature rules have to be of high order because otherwise the
numerical solution may become either inaccurate or unstable. Being aware of
this fact, we shall test the validity of our analysis of the LG methods
studied in the paper by performing some benchmark numerical tests, using
symmetric Gaussian rules of different orders to assess the influence of the
order of the quadrature rules on the accuracy and stability.

The paper is organized as follows. We make a short presentation of the
continuous problem in Section 2, and introduce the formulation and numerical
analysis of the conventional LG method for the advection equation in Section
3. Some numerical examples illustrating its performance are also reported in
this section. Section 4 is devoted to the formulation, analysis and
numerical performance of the LPS-LG method. The DC-LG method is introduced
in Section 5, studying its stability and convergence. We also present in
this section several numerical tests. Some concluding remarks are written in
Section 6.

We introduce some notation about the functional spaces used in the paper.
For $s\geq 0$ real and real $1\leq p\leq \infty $, $W^{s,p}(D)$ denotes the
real Sobolev spaces defined on $D$ for scalar real-valued functions. $%
\left\Vert \cdot \right\Vert _{W^{s,p}(D)}$ and $\left\vert \cdot
\right\vert _{W^{s,p}(D)}$ denote the norm and semi-norm, respectively, of $%
W^{s,p}(D)$. When $s=0$, $W^{0,p}(D):=L^{p}(D)$. For $p=2$, the spaces $%
W^{s,2}(D)$ are denoted by $H^{s}(D)$, which are real Hilbert spaces with
inner product $(\cdot ,\cdot )_{s}$. For $s=0$, $H^{0}(D):=L^{2}(D)$, the
inner product in $L^{2}(D)$ is denoted by $(\cdot ,\cdot )$. $H_{0}^{1}(D)$
is the space of functions of $H^{1}(D)$ which vanish on the boundary $%
\partial D$ in the sense of trace. $H^{-1}$ denotes the dual space of $%
H_{0}^{1}(D) $.\ The corresponding spaces of real vector-valued functions, $%
v:D\rightarrow \mathbb{R}^{d}$ are denoted by $W^{s,p}(D)^{d}:=\{v:D%
\rightarrow \mathbb{R}^{d}:\ v_{i}\in W^{s,p}(D),\ 1\leq i\leq d\}$. Let $X$
be a real Banach space $(X,\left\Vert \cdot \right\Vert _{X})$, if $%
v:(0,T)\rightarrow X$ is a strongly measurable function with values in $X$,
we set $\left\Vert v\right\Vert _{L^{p}(0,t;X)}=\left(
\int_{0}^{t}\left\Vert v(\tau )\right\Vert _{X}^{p}d\tau \right) ^{1/p}$ for
$1\leq p<\infty $, and $\left\Vert v\right\Vert _{L^{\infty }(0,t;X)}= %
\displaystyle \text{ess}\sup_{0<\tau \leq t}\left\Vert v(\tau )\right\Vert
_{X}$; when $t=T$, we shall write, unless otherwise stated, $\left\Vert
v\right\Vert _{L^{p}(X)}$. We shall also use the following discrete norms:%
\begin{equation*}
\left\Vert v\right\Vert _{l^{p}(X)}=\left( \Delta t\sum_{i=1}^{N}\left\Vert
v(\tau _{i})\right\Vert _{X}^{p}\right) ^{1/p}\text{, \ }\left\Vert
v\right\Vert _{l^{\infty }(X)}=\max_{1\leq i\leq N}\left\Vert v(\tau
_{i})\right\Vert _{X},
\end{equation*}%
corresponding to the time discrete space $l^{p}(X)\equiv l^{p}(0,T;X)$, $%
1\leq p<\infty $, defined as
\begin{equation*}
l^{p}(0,T;X):=\left\{ v:(0,t_{1},t_{2},\ldots ,t_{N}=T)\rightarrow X:\ \
\left\Vert v\right\Vert _{l^{p}(X)}<\infty \right\} ,
\end{equation*}%
when $p=\infty$
\begin{equation*}
l^{\infty }(0,T;X):=\left\{ v:(0,t_{1},t_{2},\ldots ,t_{N}=T)\rightarrow X:\
\max_{1\leq i\leq N}\left\Vert v(\tau _{i})\right\Vert _{X}<\infty \right\}.
\end{equation*}%
Finally, we shall also use the space of continuous functions such as $%
C^{r}(D )$ that denotes the space of $r$-times continuously differentiable
functions on $D$, when $r=0$ we write $C(D)$ instead of $C^{0}(D)$; the
space $C^{r,1}(\overline{D})$, $r\geq 0$, of functions defined on the
closure of $D$, $r$ -times continuously differentiable and with the $r$th
derivative being Lipschitz continuous; and the space of continuous and
bounded functions in time with values in $X$ denoted by $C([0,T];X)$.

Throughout this paper, $C$ will denote a generic positive constant which is
independent of both the space and time discretization parameters $h$ and $%
\Delta t$ respectively. $C$ will have different values at different places
of appearance. In many places we shall use, without making any explicit
statement, the Cauchy's inequality $ab\leq \frac{\epsilon }{2}a^{2}+\frac{1}{%
2\epsilon }b^{2}\ (a,\ b>0,\ \epsilon >0)$, and the discrete Gronwall
inequality presented in \cite{HR}.

\section{The Cauchy problem for the advection equation}

To introduce the LG method we consider the Cauchy problem for the first
order linear hyperbolic equation%
\begin{equation}
\left\{
\begin{array}{l}
\dfrac{\partial c}{\partial t}+\mathbf{u\cdot }\nabla c=0,\ \ x\in \mathbb{R}%
^{d},\ t>0,\  \\
\\
c(x,0)=c^{0}(x),%
\end{array}%
\right.  \label{eq1}
\end{equation}%
where $c:\mathbb{R}^{d}\times \lbrack 0,T]\rightarrow \mathbb{R}$, $\mathbf{u%
}:\mathbb{R}^{d}\times \lbrack 0,T]\rightarrow \mathbb{R}^{d}$ is a
vector-valued function and $c^{0}(x)$ is a function of compact support
defined in a domain $D_{0}\subset \mathbb{R}^{d}$. Considering the
characteristics curves of the first order differential operator $%
D/Dt:=\partial /\partial t+\mathbf{u\cdot }\nabla $ which are the solution
to the system of ordinary differential equations
\begin{equation}
\left\{
\begin{array}{l}
\dfrac{dX\mathbf{(}x\mathbf{,}s;t)}{dt}=\mathbf{u}(X(x,s;t),t), \\
\\
X\mathbf{(}x,s;s)=x\mathbf{,}%
\end{array}%
\right.  \label{eq2}
\end{equation}%
we can recast problem (\ref{eq1}) as an ordinary differential equation along
the characteristics curves, $X\mathbf{(}x\mathbf{,}s;t)$, of the form%
\begin{equation}
\left\{
\begin{array}{l}
\dfrac{Dc}{Dt}=0\text{ ,\ \ }X\mathbf{(}x\mathbf{,}s;t)\in \mathbb{R}^{d}%
\text{,\ }t>0,\text{\ } \\
\\
c(X\mathbf{(}x\mathbf{,}0;0),0)=c^{0}(x\mathbf{)}.%
\end{array}%
\right.  \label{eq3}
\end{equation}%
Assuming that $\mathbf{u}\in C([0,T],W^{1,\infty }(\mathbb{R}^{d})^{d})$, so
problem (\ref{eq2}) has a unique solution, and $c^{0}(x)$ is sufficiently
smooth, we have that the solution of (\ref{eq3}) is then given by%
\begin{equation}
c(X\mathbf{(\cdot ,\ }t;t+\tau ),t+\tau )=c(\cdot ,t).  \label{eq4}
\end{equation}%
Concerning the solution $t\rightarrow X(x,s;t)$ to (\ref{eq2}), the
following regularity results are in order.

\begin{lemma}
\label{solX copy(1)} \textit{Assume that } $\mathbf{u}\in
C([0,T],W^{k,\infty }(\mathbb{R} ^{d})^{d})$, \ $k\geq 1$. \textit{\ Then
for }$s,\ t\in \lbrack 0,T]$\textit{, there exists a unique solution }$%
t\rightarrow X(x,s;t)\ $\textit{of (\ref{eq2}), such that }$X(x,s;t)\in
W^{1,\infty }(W^{k,\infty }( \mathbb{R} ^{d})^{d})$\textit{. Furthermore,
let the multi-index }$\alpha \in N^{d}$\textit{, then for all }$\alpha ,$%
\textit{\ such that }$1\leq \mid \alpha \mid \leq k,$\textit{\ }$D^{\mathbf{%
\alpha }}X_{i}(x,s;t)\in C([0,T],L^{\infty }(\mathbb{R}^{d}\times \lbrack
0,T]))$, $1\leq i\leq d$.
\end{lemma}

Next, we consider the mapping $\varphi _{s}^{t}:\mathbb{R}^{d}\rightarrow
\mathbb{R}^{d}$, defined by $\varphi _{s}^{t}(x)=X(x,s;t)$, since $%
X(X(x,s;t),t;s)=x$, then it follows that the mapping $\varphi _{t}^{s}$ is
the inverse of $\varphi _{s}^{t}$. The Jacobian determinant of this
transformation%
\begin{equation}
J(x,s;t)=\det \left( \frac{\partial X_{i}(x,s;t)}{\partial x_{j}}\right) ,\
1\leq i,j\leq d,  \label{eqj1}
\end{equation}%
satisfies the equation
\begin{equation}
\frac{\partial J(x,s;t)}{\partial t}=J(x,s;t)\mathrm{div\ }\mathbf{u(}%
X(x,s;t),t).  \label{eqj2}
\end{equation}%
It is easy to see that if $C_{\mathbf{u}}:=\left\Vert \text{\textrm{div}}\
\mathbf{u}\right\Vert _{L^{\infty }(D\times (0,T))}$, then%
\begin{equation}
\exp (-C_{\mathbf{u}}\left\vert s-t\right\vert )\leq J(x,s;t)\leq \exp (C_{%
\mathbf{u}}\left\vert s-t\right\vert ).  \label{jac}
\end{equation}%
Moreover, for $\left\vert t-s\right\vert $ sufficiently small it follows that%
\begin{equation}
K_{1}\mid x-y\mid \leq \mid X(x,s;t)-X(y,s;t)\mid \leq K_{2}\mid x-y\mid ,
\label{jac2}
\end{equation}%
where $K_{1}=(1-\mid s-t\mid \cdot \mid \nabla \mathbf{u}\mid _{L^{\infty
}(L^{\infty }(D)^{d\times d})})$, and $K_{2}=\exp (\mid s-t\mid \cdot \mid
\nabla \mathbf{u}\mid _{L^{\infty }(L^{\infty }(D)^{d\times d})})$. Here, $%
\mid a-b\mid $\textit{\ }denotes the Euclidean distance between the points%
\textit{\ }$a,\ b\in \mathbb{R}^{d}$. Hereafter, for the sake of simplicity,
we make the assumption $\mathrm{div\ }\mathbf{u}=0$. An important
consequence of this assumption is that $J(x,s;t)=1$ almost everywhere in $%
\mathbb{R}$. However, we must remark that one can easily accommodate the
proofs of \ our results to the general case of $\mathrm{div\ }\mathbf{u}\neq
0$.

\section{The conventional LG method for the advection equation}

In the framework of finite elements, Douglas and Russell (1982) and
Pironneau (1982) proposed the so called conventional LG method as a time
marching algorithm to approximate the solution of (\ref{eq1}).

\subsection{Finite element formulation}

The realization of this method requires the definition of a family of
partitions $D_{h}$ in a domain $D\subset \mathbb{R}^{d}$ sufficiently large,
such that given $T>0$, $D_{0}\subset \subset D$ and for all $t\in [0,T]$ we
can assume that $c(x,t)=0$ on the boundary $\partial D$. The partitions $%
D_{h}$ generated in the closed region $\overline{D}:=D\cup \partial D$ are
quasi-uniform regular and composed of $d$-simplices $K$, the boundaries of
which are denoted by $\partial K$. $h_{K}$ denotes the diameter of $K$ and
the mesh parameter $h:=\max_{K}h_{K}$. Moreover, we shall assume that $%
\mathbf{u(}x,t\mathbf{)}$ is zero on the boundary $\partial D$. To define
the finite element spaces we use the reference simplex $\widehat{K}$ with
vertices $\left\{ \widehat{x}_{i}\right\} _{i=1}^{d+1}$, $\widehat{K}%
:=\left\{ \widehat{x}\in \mathbb{R}^{d}:0\leq \widehat{x}_{i}\leq 1,\
1-\sum_{i=1}^{d}\widehat{x}_{i}\geq 0\right\} $, such that for each $K\in
D_{h}$ there is an invertible affine mapping $F_{K}:\widehat{K}\rightarrow
K, $
\begin{equation*}
F_{K}(\widehat{x})=\mathbf{B}_{K}\widehat{x}+\mathbf{b}_{K},\ \ \mathbf{B}%
_{K}\in \mathcal{L}(\mathbb{R}^{d})\ \text{and}\ \mathbf{b}_{K}\in \mathbb{R}%
^{d}.
\end{equation*}%
The finite element spaces used in the formulation of the LG method are the
following:
\begin{equation*}
W_{h}:=\left\{ v_{h}\in C(\overline{D}):\forall K\in D_{h},\ v_{h}\mid
_{K}\in P_{m}(K)\right\} ,\ \text{and}\ V_{h}=H_{0}^{1}(D)\cap W_{h},
\end{equation*}%
with%
\begin{equation*}
P_{m}(K)=\left\{ p(x):\text{for }x\in K\text{, }p(x)=\widehat{p}\circ
F_{K}^{-1}(x)\text{, }\widehat{p}\in P_{m}(\widehat{K})\right\} ,
\end{equation*}%
where $P_{m}(\widehat{K})$ denotes the set of polynomials of degree $\leq m$
defined in $\widehat{K}$.\ Next, we introduce some auxiliary results
concerning the approximation properties of the finite element spaces. For $%
0<h<h_{0}<1$, there exists a constant $c_{1}$ independent of $h$ such that
for $w\in H^{q+1}(D)\cap H_{0}^{1}(D)$ and $1\leq q\leq m$,
\begin{equation}
\inf_{v_{h}\in V_{h}}\left\{ \left\Vert w-v_{h}\right\Vert +h\left\Vert
\nabla \left( w-v_{h}\right) \right\Vert \right\} \leq
c_{1}h^{q+1}\left\vert w\right\vert _{H^{q+1}(D)}\text{.}  \label{eq4a}
\end{equation}%
Since the partition $D_{h}$ is quasi-uniformly regular, the following
inverse inequality holds: for all $w_{h}\in W_{h}$ and $0\leq k\leq l\leq 1$%
, and $1\leq p\leq q\leq \infty $, there exists a constant $c_{\mathrm{inv}}$
independent of $h$ such that,%
\begin{equation}
\left\Vert w_{h}\right\Vert _{W^{l,q}(D)}\leq c_{\mathrm{inv}%
}h^{d/q-d/p+k-m}\left\Vert w_{h}\right\Vert _{W^{k,p}(D)}.  \label{eq4b}
\end{equation}%
Let $\Pi _{h}:C(D)\rightarrow W_{h}$ be the Lagrange interpolation operator
in $W_{h}$ and let $P_{h}:L^{2}(D)\rightarrow V_{h}$ be the orthogonal $%
L^{2} $-projector defined as%
\begin{equation}
\left( w-P_{h}w,v_{h}\right) =0\text{ \ for all }v_{h}\in V_{h}\text{,}
\label{eq4c}
\end{equation}%
then there are constants $c_{2}$ and $c_{3}$ independent of $h$, such that
for $0\leq \sigma \leq m$, and $1\leq \gamma \leq \infty $,%
\begin{equation}
\left\Vert w-P_{h}w\right\Vert +h\left\Vert \nabla (w-P_{h}w)\right\Vert
\leq c_{2}h^{\sigma +1}\left\vert w\right\vert _{H^{\sigma +1}(D)}
\label{eq4d}
\end{equation}%
and%
\begin{equation}
\left\Vert w-\Pi _{h}w\right\Vert _{L^{\gamma }(D)}+h\left\Vert \nabla
(w-\Pi _{h}w)\right\Vert _{L^{\gamma }(D)}\leq c_{3}h^{\sigma +1}\left\vert
w\right\vert _{W^{\sigma +1,\gamma }(D)},  \label{eq4e}
\end{equation}%
respectively \cite{Ci}. It is worth noting that the estimate (\ref{eq4e})
and the inverse inequality are also valid when the domain $D$ is substituted
by an element $K $. The following properties of the projector $P_{h}$ are
also used in the paper:%
\begin{equation*}
P_{h}v_{h}=v_{h}\ \forall v_{h}\in V_{h},
\end{equation*}%
and (contractiveness)%
\begin{equation*}
\left\Vert P_{h}v\right\Vert \leq \left\Vert v\right\Vert \ \forall v\in
L^{2}(D).
\end{equation*}

Let $\mathcal{P}:=$ $0=t_{0}<t_{1}<\ldots <t_{N}=T$ be a uniform partition
of step length $\Delta t$ for the interval $[0,T]$, the finite element
solution of (\ref{eq1}) at time $t_{n}$, denoted by $c_{h}^{n}\in V_{h}$, is
given by
\begin{equation*}
c_{h}^{n}=\sum_{i=1}^{N_{P}}C_{i}^{n}\phi _{i},
\end{equation*}%
where $C_{i}^{n}:=c_{h}^{n}(x_{i})$, $x_{i}$ being the $i$th mesh-point in $%
D_{h}$, $N_{P}$ denotes the number of mesh-points of the partition $D_{h}$,
and $\{\phi _{i}\}_{i=1}^{N_{P}}$ is the set of global basis functions of $%
V_{h}$. The conventional LG method calculates $c_{h}^{n}\in V_{h}$ as%
\begin{equation}
c_{h}^{n}(x)=P_{h}c_{h}^{n-1}\circ X(x,t_{n};t_{n-1}),  \label{cua40}
\end{equation}%
or equivalently,\ for all $i=1,...,N_{P},$

\begin{equation}
\int_{D}c_{h}^{n}(x)\phi _{i}(x)dx=\int_{D}c_{h}^{n-1}\circ
X(x,t_{n};t_{n-1})\phi _{i}(x)dx,  \label{ecua4}
\end{equation}%
where $X(x,t_{n};t_{n-1})$ is the position at time instant $t_{n-1}$ of the
point that at time instant $t_{n}$ is located at the point $x$.

\textbf{Notations }Let us introduce some shorthand notations in order to
simplify the writing of the formulas that will appear in the article. In the
sequel, we sometimes use $X^{n,n-1}$ or if confusion may arise $%
X^{n,n-1}(x), $ to denote $X(x,t_{n};t_{n-1})$. Also, let $a(x,t)$ be a
generic function defined in $\mathbb{R} ^{d}\times \left[ 0,T\right] $, then
$a^{n}(x)$ will denote the value of $a(x,t)$ at time instant $t_{n},$ that
is, $a^{n}(x)=a(x,t_{n})$, whereas $a^{\ast n-1}(x)$ denotes $a^{n-1}\circ
X^{n,n-1}(x)$.

Hereafter, we assume that $h\in (0,h_{0})$ and $\Delta t\in (0,\Delta t_{0})$
with $h_{0}<1$ and $\Delta t_{0}<1$.

%

\subsection{Analysis of the conventional LG method}

We begin analyzing the $L^{2}$-norm stability of the method.

\begin{lemma}
\label{stability} For all $N\geq 1$,
\begin{equation}
\left\Vert c_{h}^{N}\right\Vert ^{2}+\sum_{n=1}^{N}\left\Vert
c_{h}^{n}-c_{h}^{\ast n-1}\right\Vert ^{2}=\left\Vert c_{h}^{0}\right\Vert
^{2}.  \label{eq6}
\end{equation}
\end{lemma}

\begin{proof}
First of all, we show that for any function $f^{\ast
n-1}(x):=f^{n-1}(X(x,t_{n},t_{n-1}))\in L^{2}(D)$%
\begin{equation}
\left\Vert f^{\ast n-1}\right\Vert =\left\Vert f^{n-1}\right\Vert .
\label{eq6.0}
\end{equation}%
To see this is so, we make the change of variable $y=X^{n,n-1}(x)$ and
recall that the Jacobian determinant of this transformation, $%
J(x,t_{n};t_{n-1})=1$ a.e., then
\begin{equation*}
\int_{D}\left\vert f^{\ast n-1}(x)\right\vert ^{2}dx=\int_{D}\left\vert
f^{n-1}(X^{n,n-1}(x))\right\vert ^{2}dx=\int_{D}\left\vert
f^{n-1}(y)\right\vert ^{2}J^{-1}dy=\int_{D}\left\vert f^{n}(y)\right\vert
^{2}dy.
\end{equation*}%
Now, we notice that from (\ref{cua40}) it follows that%
\begin{equation*}
\left( c_{h}^{n}-c_{h}^{\ast n-1},c_{h}^{n}\right) =0,
\end{equation*}%
then using the elementary relation $2(a-b,a)=a^{2}+(a-b)^{2}-b^{2},\ a,b\in
\mathbb{R}$, we obtain that
\begin{equation*}
2\left( c_{h}^{n}-c_{h}^{\ast n-1},c_{h}^{n}\right) =\left\Vert
c_{h}^{n}\right\Vert ^{2}+\left\Vert c_{h}^{n}-c_{h}^{\ast n-1}\right\Vert
^{2}-\left\Vert c_{h}^{\ast n-1}\right\Vert ^{2},
\end{equation*}%
and by virtue of (\ref{eq6.0})%
\begin{equation*}
\left\Vert c_{h}^{n}\right\Vert ^{2}+\left\Vert c_{h}^{n}-c_{h}^{\ast
n-1}\right\Vert ^{2}-\left\Vert c_{h}^{n-1}\right\Vert ^{2}=0.
\end{equation*}%
Summing this expression from $n=1$ up to $n=N$ it follows (\ref{eq6}).
\end{proof}

\begin{remark}
Following \cite{john}, we can interpret the term $\sum_{n=1}^{N}\left\Vert
c_{h}^{n}-c_{h}^{\ast n-1}\right\Vert ^{2}$ as a measure of the numerical
dissipation of conventional LG methods. It is shown there that
\begin{equation*}
\sum_{n=1}^{N}\left\Vert c_{h}^{n}-c_{h}^{\ast n-1}\right\Vert ^{2}\leq C%
\frac{h^{4}}{\Delta t}\sum_{n=1}^{N}\Delta t\left\Vert \Delta
_{h}^{n-1}c_{h}^{\ast n-1}\right\Vert^{2} ,
\end{equation*}%
where $\Delta _{h}^{n-1}:H^{1}(D)\rightarrow W_{h}^{\ast n-1}:=\left\{
v_{h}^{\ast n-1}(x)=v_{h}^{n-1}(X^{n,n-1}(x)):v_{h}^{n-1}(x)\in
W_{h}\right\} $ denotes the discrete Laplacian operator. When $\Delta t=h$
this amounts to adding an artificial diffusion term to the continuous
advection equation of the form $Ch^{3}\Delta c$; so, for sufficiently smooth
solutions such an artificial diffusion is not excessive, in particular if
one compares with the usual upwind method that adds an artificial diffusion
therm of the form $-Ch\Delta c$, but it may be insufficient to eliminate the
oscillations when the exact solution is not smooth. To deal with the case of
non smooth solutions we introduce the DC-LG method.
\end{remark}

The remainder of this section is devoted to the analysis of the
convergence.\ We have the following result.

\begin{theorem}
\label{teorema1} Let $c\in L^{\infty }(H^{m+1}(D)\cap H_{0}^{1}(D))$. Then
there exists a constant $C$ independent of $\Delta t,\ h,$ and $n$, such that%
\begin{equation}
\left\Vert c-c_{h}\right\Vert _{l^{\infty }\left( L^{2}(D)\right) }\leq
\left\Vert e^{0}\right\Vert +C\min \left( 1,\displaystyle\frac{\Delta
t^{1/2}\left\Vert \mathbf{u}\right\Vert _{L^{\infty }(W^{1,\infty }(D)^{d}))}%
}{h}\right) \displaystyle\frac{h^{m+1}}{\Delta t^{1/2}}\left\vert
c\right\vert _{L^{\infty }(H^{m+1}(D))}.  \label{eq8}
\end{equation}
\end{theorem}

\begin{proof}
The error $e^{n}:=c^{n}-c_{h}^{n}$ can be expressed as%
\begin{equation}
e^{n}=\left( c^{n}-P_{h}c^{n}\right) +\left( P_{h}c^{n}-c_{h}^{n}\right)
\equiv \rho ^{n}+\theta _{h}^{n},  \label{eq9}
\end{equation}%
where $\theta _{h}^{n}\in V_{h}$. Noting that $P_{h}\rho ^{n}=0$, then it
follows that%
\begin{equation*}
\left\Vert e^{n}\right\Vert ^{2}=(\rho ^{n}+\theta _{h}^{n},\rho ^{n}+\theta
_{h}^{n})=\left\Vert \rho ^{n}\right\Vert ^{2}+\left\Vert \theta
_{h}^{n}\right\Vert ^{2}.
\end{equation*}%
By virtue of (\ref{eq4d}), $\rho ^{n}$ satisfies the bound%
\begin{equation}
\left\Vert \rho ^{n}\right\Vert _{H^{l}(D)}\leq c_{2}h^{m+1-l}\left\vert
c^{n}\right\vert _{H^{m+1}(D)}\text{\ \ }(0\leq l\leq m+1).  \label{eq9.1}
\end{equation}%
To estimate $e^{n}$ we make use of (\ref{eq4}), which implies that $%
c^{n}=c^{\ast n-1}$, so that $\forall v_{h}\in V_{h}$%
\begin{equation*}
(c^{n}-c^{\ast n-1},v_{h})=0,\
\end{equation*}%
so, subtracting (\ref{cua40}) from this equation it results the following
error equation%
\begin{equation}
\left( e^{n}-{e}^{\ast n-1},v_{h}\right) =0,  \label{eq10}
\end{equation}%
where ${e}^{\ast
n-1}(x)=c^{n-1}(X^{n,n-1}(x))-c_{h}^{n-1}(X^{n,n-1}(x))=\rho ^{n-1}\left(
X^{n,n-1}(x)\right) +\theta _{h}^{n-1}\left( X^{n,n-1}(x)\right) $. Now, we
calculate an error estimate from (\ref{eq10}). First, we notice that by
virtue of (\ref{eq6.0}) $\left\Vert e^{\ast n-1}\right\Vert =\left\Vert
e^{n-1}\right\Vert $, so this property together with the elementary relation
$2(a-b)a=a^{2}+(a-b)^{2}-b^{2}$, permits us to write%
\begin{equation}
\left\Vert e^{n}\right\Vert ^{2}+\left\Vert e^{n}-e^{\ast n-1}\right\Vert
^{2}-\left\Vert e^{n-1}\right\Vert ^{2}=2\left( e^{n}-e^{\ast
n-1},e^{n}\right) .  \label{err1}
\end{equation}%
Now, one needs to estimate the term $\left( e^{n}-e^{\ast n-1},e^{n}\right) $%
. For this purpose, we apply the argument of \cite{john}, use (\ref{eq9})
and set%
\begin{equation*}
(e^{n}-e^{\ast n-1},e^{n})=\left( e^{n}-e^{\ast n-1},\rho ^{n}\right)
+\left( e^{n}-e^{\ast n-1},\theta _{h}^{n}\right) ,
\end{equation*}%
but by virtue of (\ref{eq10}), $\left( e^{n}-e^{\ast n-1},\theta
_{h}^{n}\right) =0$, then the only term we have to estimate is $\left(
e^{n}-e^{\ast n-1},\rho ^{n}\right) $. Thus, by the Cauchy-Schwarz inequality%
\begin{equation}
\left( e^{n}-e^{\ast n-1},\rho ^{n}\right) \leq \frac{1}{4}\left\Vert
e^{n}-e^{\ast n-1}\right\Vert ^{2}+\left\Vert \rho ^{n}\right\Vert ^{2};
\label{er1}
\end{equation}%
hence, one can write that%
\begin{equation*}
2\left( e^{n}-e^{\ast n-1},e^{n}\right) \leq \displaystyle\frac{1}{2}%
\left\Vert e^{n}-e^{\ast n-1}\right\Vert ^{2}+2\left\Vert \rho
^{n}\right\Vert ^{2}.
\end{equation*}%
Using this bound on the right hand side of (\ref{err1}) it follows that%
\begin{equation*}
\left\Vert e^{n}\right\Vert ^{2}+\frac{1}{2}\left\Vert e^{n}-e^{\ast
n-1}\right\Vert ^{2}-\left\Vert e^{n-1}\right\Vert ^{2}\leq 2\left\Vert \rho
^{n}\right\Vert ^{2}.
\end{equation*}%
From this expression, summing from $n=1$ up to $n=N$ one readily obtains that%
\begin{equation*}
\left\Vert e^{N}\right\Vert ^{2}+\frac{1}{2}\sum_{n=1}^{N}\left\Vert
e^{n}-e^{\ast n-1}\right\Vert ^{2}\leq \frac{2T}{\Delta t}\left\Vert \rho
\right\Vert _{l^{\infty }(L^{2}(D))}^{2}+\left\Vert e^{0}\right\Vert ^{2}.
\end{equation*}%
Then using (\ref{eq9.1}) yields%
\begin{equation}
\left\Vert c-c_{h}\right\Vert _{l^{\infty }\left( L^{2}(D)\right) }\leq C%
\displaystyle\left( \frac{h^{m+1}}{\Delta t^{1/2}}\left\vert c\right\vert
_{L^{\infty }(H^{m+1}(D))}\right) +\left\Vert e^{0}\right\Vert.  \label{10a3}
\end{equation}%
For $\Delta t=O(h)$, the error is $O(h^{m+1/2})$, which is of the same order
as the streamline-diffusion method \cite{JNP} for the advection equation.
However, this estimate does not allow the convergence of the method when $%
\Delta t\rightarrow 0$ independently of $h$. To overcome this trouble, we
apply the procedure of \cite{MS} to obtain an error estimate valid for all $%
\Delta t$. So, substituting $e^{n}=\rho ^{n}+\theta _{h}^{n}$,$\ $and$\
e^{\ast n-1}=\rho ^{\ast n-1}+\theta _{h}^{\ast n-1}$ in (\ref{eq10}) and
rearranging terms yields

\begin{equation}
\left( \theta _{h}^{n}-\theta _{h}^{\ast n-1},v_{h}\right) =-(\rho ^{n}-\rho
^{n-1},v_{h})-(\rho ^{n-1}-\rho ^{\ast n-1},v_{h}).  \label{11}
\end{equation}%
Letting $v_{h}=\theta _{h}^{n}$ we bound each term of this equality as
follows. First, we notice that $\left\Vert \theta _{h}^{\ast n-1}\right\Vert
^{2}=\left\Vert \theta _{h}^{n-1}\right\Vert ^{2}$ and consequently%
\begin{equation*}
\left\Vert \theta _{h}^{n}\right\Vert ^{2}+\left\Vert \theta _{h}^{n}-\theta
_{h}^{\ast n-1}\right\Vert ^{2}-\left\Vert \theta _{h}^{n-1}\right\Vert
^{2}=2\left( \theta _{h}^{n}-\theta _{h}^{\ast n-1},\theta _{h}^{n}\right) .
\end{equation*}%
Second, since for each $n$, $P_{h}\rho ^{n}=0$, then it follows that $\
(\rho ^{n}-\rho ^{n-1},\theta _{h}^{n})=P_{h}(\rho ^{n}-\rho ^{n-1})=0$. It
remains to bound the term $2(\rho ^{n-1}-\rho ^{\ast n-1},\theta _{h}^{n})$.
To do so, we notice that%
\begin{equation*}
\rho ^{n-1}-\rho ^{\ast n-1}=\int_{t_{n-1}}^{t_{n}}\frac{d\rho
(X(x,t_{n};t),t_{n-1})}{dt}dt,
\end{equation*}%
since $\displaystyle \frac{d\rho (X(x,t_{n};t),t_{n-1})}{dt}=\mathbf{u}%
(X(x,t_{n};t),t)\cdot \nabla _{X}\rho (X(x,t_{n};t),t_{n-1})$, then by the
Cauchy-Schwarz inequality we get%
\begin{equation*}
\left\vert \rho ^{n-1}-\rho ^{\ast n-1}\right\vert ^{2}\leq \Delta
t\int_{t_{n-1}}^{t_{n}}\left\vert \mathbf{u}(X(x,t_{n};t),t)\cdot \nabla
_{X}\rho (X(x,t_{n};t),t_{n-1})\right\vert ^{2}dt,
\end{equation*}%
so, letting $y=X(x,t_{n};t)$ and denoting by $J^{t,n}$ the Jacobian
determinant $J(x,t;t_{n}):=\left( \frac{\partial X (x,t;t_{n})}{\partial x}%
\right) =1$, it follows that%
\begin{equation*}
\begin{array}{r}
\left\Vert \rho ^{n-1}-\rho ^{\ast n-1}\right\Vert ^{2}\leq \Delta t%
\displaystyle\int_{D}\int_{t_{n-1}}^{t_{n}}\left\vert \mathbf{u}(y,t)\cdot
\nabla \rho ^{n-1}(y)\right\vert ^{2}\left( J^{t,n-1}\right) ^{-1}dtdy\  \\
\\
\leq \Delta t\left\Vert \mathbf{u}\right\Vert _{L^{\infty }(L^{\infty
}(D)^{d})}^{2}\displaystyle\int_{t_{n-1}}^{t_{n}}\int_{D}\left\vert \nabla
\rho ^{n-1}(y)\right\vert ^{2}dydt\text{ } \\
\\
\leq \Delta t^{2}\left\Vert \mathbf{u}\right\Vert _{L^{\infty }(L^{\infty
}(D)^{d})}^{2}\left\Vert \nabla \rho ^{n-1}\right\Vert ^{2}.%
\end{array}%
\end{equation*}%
Now, using the estimate (\ref{eq9.1}) we can write that%
\begin{equation}
\left\Vert \rho ^{n-1}-\rho ^{\ast n-1}\right\Vert ^{2}\leq \Delta
t^{2}C\left( \frac{\Delta t^{1/2}\left\Vert \mathbf{u}\right\Vert
_{L^{\infty }(L^{\infty }(D)^{d})}}{h}\right) ^{2}\left( \frac{h^{m+1}}{%
\Delta t^{1/2}}\right) ^{2}\left\vert c\right\vert _{L^{\infty
}(H^{m+1}(D))}^{2}.  \label{roro*}
\end{equation}%
Then%
\begin{equation*}
\begin{array}{l}
2(\rho ^{n-1}-\rho ^{\ast n-1},\theta _{h}^{n})\leq \displaystyle\frac{2}{%
\Delta t}\left\Vert \rho ^{n-1}-\rho ^{\ast n-1}\right\Vert ^{2}+%
\displaystyle\frac{\Delta t}{2}\left\Vert \theta _{h}^{n}\right\Vert ^{2} \\
\\
\leq \Delta tC\displaystyle\left( \frac{\Delta t^{1/2}\left\Vert \mathbf{u}%
\right\Vert _{L^{\infty }(L^{,\infty }(D)^{d})}}{h}\right) ^{2}\left( \frac{%
h^{m+1}}{\Delta t^{1/2}}\right) ^{2}\left\vert c\right\vert _{L^{\infty
}(0,T;H^{m+1}(D))}^{2}+\displaystyle\frac{\Delta t}{2}\left\Vert \theta
_{h}^{n}\right\Vert ^{2}.%
\end{array}%
\end{equation*}%
Collecting all these bounds we have that%
\begin{equation*}
\left\Vert \theta _{h}^{n}\right\Vert ^{2}+\left\Vert \theta _{h}^{n}-\theta
_{h}^{\ast n-1}\right\Vert ^{2}-\left\Vert \theta _{h}^{n-1}\right\Vert
^{2}\leq \Delta tC\displaystyle\left( \frac{\Delta t^{1/2}\left\Vert \mathbf{%
u}\right\Vert _{L^{\infty }(L^{\infty }(D)^{d})}}{h}\right) ^{2}\left( \frac{%
h^{m+1}}{\Delta t^{1/2}}\right) ^{2}\left\vert c\right\vert _{L^{\infty
}(H^{m+1}(D))}^{2}+\frac{\Delta t}{2}\left\Vert \theta _{h}^{n}\right\Vert
^{2}.
\end{equation*}%
Now, summing from $n=1$ up to $n=N$ yields%
\begin{equation*}
\left\Vert \theta _{h}^{N}\right\Vert ^{2}+\sum_{n=1}^{N}\left\Vert \theta
_{h}^{n-1}-\theta _{h}^{\ast n-1}\right\Vert ^{2}\leq \left\Vert \theta
_{h}^{0}\right\Vert ^{2}+R^{2}+\frac{\Delta t}{2}\sum_{n=1}^{N}\left\Vert
\theta _{h}^{n}\right\Vert ^{2},
\end{equation*}%
where%
\begin{equation*}
R^{2}=\displaystyle C\left( \frac{\Delta t^{1/2}\left\Vert \mathbf{u}%
\right\Vert _{L^{\infty }(L^{\infty }(D)^{d})}}{h}\right) ^{2}\left( \frac{%
h^{m+1}}{\Delta t^{1/2}}\right) ^{2}\left\vert c\right\vert _{L^{\infty
}(H^{m+1}(D))}^{2}.
\end{equation*}%
Since $\left\Vert e^{N}\right\Vert ^{2}=\left\Vert \rho ^{N}\right\Vert
^{2}+\left\Vert \theta _{h}^{N}\right\Vert ^{2}$ and $\sum_{n=1}^{N}\left%
\Vert \theta _{h}^{n}\right\Vert ^{2}+\left\Vert \rho ^{N}\right\Vert
^{2}\leq \sum_{n=1}^{N}\left\Vert e^{n}\right\Vert ^{2}$, then it follows
that
\begin{equation*}
\left\Vert e^{N}\right\Vert ^{2}+\sum_{n=1}^{N}\left\Vert \theta
_{h}^{n-1}-\theta _{h}^{\ast n-1}\right\Vert ^{2}\leq \left\Vert
e^{0}\right\Vert ^{2}+R^{2}+\Delta t\sum_{n=1}^{N}\left\Vert
e^{n}\right\Vert ^{2}.
\end{equation*}%
Applying Gronwall inequality yields%
\begin{equation}
\left\Vert c-c_{h}\right\Vert _{l^{\infty }\left( L^{2}(D)\right) }\leq
C_{1}\left( \frac{\Delta t^{1/2}\left\Vert \mathbf{u}\right\Vert _{L^{\infty
}(L^{\infty }(D)^{d})}}{h}\right) \left( \frac{h^{m+1}}{\Delta t^{1/2}}%
\right) \left\vert c\right\vert _{L^{\infty }(0,T;H^{m+1}(D))}+\left\Vert
e^{0}\right\Vert .  \label{12}
\end{equation}%
Since (\ref{10a3}) is valid, then combining it with (\ref{12}) yields the
result (\ref{eq8}).
\end{proof}

\subsection{Numerical test with the conventional LG method}

We study the behavior of the conventional LG method considering the rotating
hump problem \cite{jack}. The domain $D:=(-1,1)\times (-1,1)$, the velocity
field is $\mathbf{u}=2\pi (-x_{2},x_{1})$, and the initial condition%
\begin{equation}
c^{0}(x)=\left\{
\begin{array}{l}
\cos ^{3}(\frac{3}{2}\pi r),\ \ r\leq 1/3, \\
\\
0\ \ \ \ \ \ \ \ \ \ \ \ \ \ \text{otherwise,}%
\end{array}%
\right.  \label{ic1}
\end{equation}%
where $r^{2}=(x_{1}-0.5)^{2}+x_{2}^{2}$. Notice that the function $\cos
^{p}(3\pi r/2)\in H^{p}(D),p\geq 1$, then $p=3$ allows enough smoothness for
the optimal estimate of the error when $m=1$ and $2$. We show in Figure 1
the isolines of the $L^{2}$-projected initial condition in a mesh with mesh
parameter $h=0.05$ and the cross section of the exact initial condition $%
c^{0}(x)$ at $x_{2}=0$.
\begin{figure}[th!]
\begin{center}
\scalebox{0.45}{\includegraphics{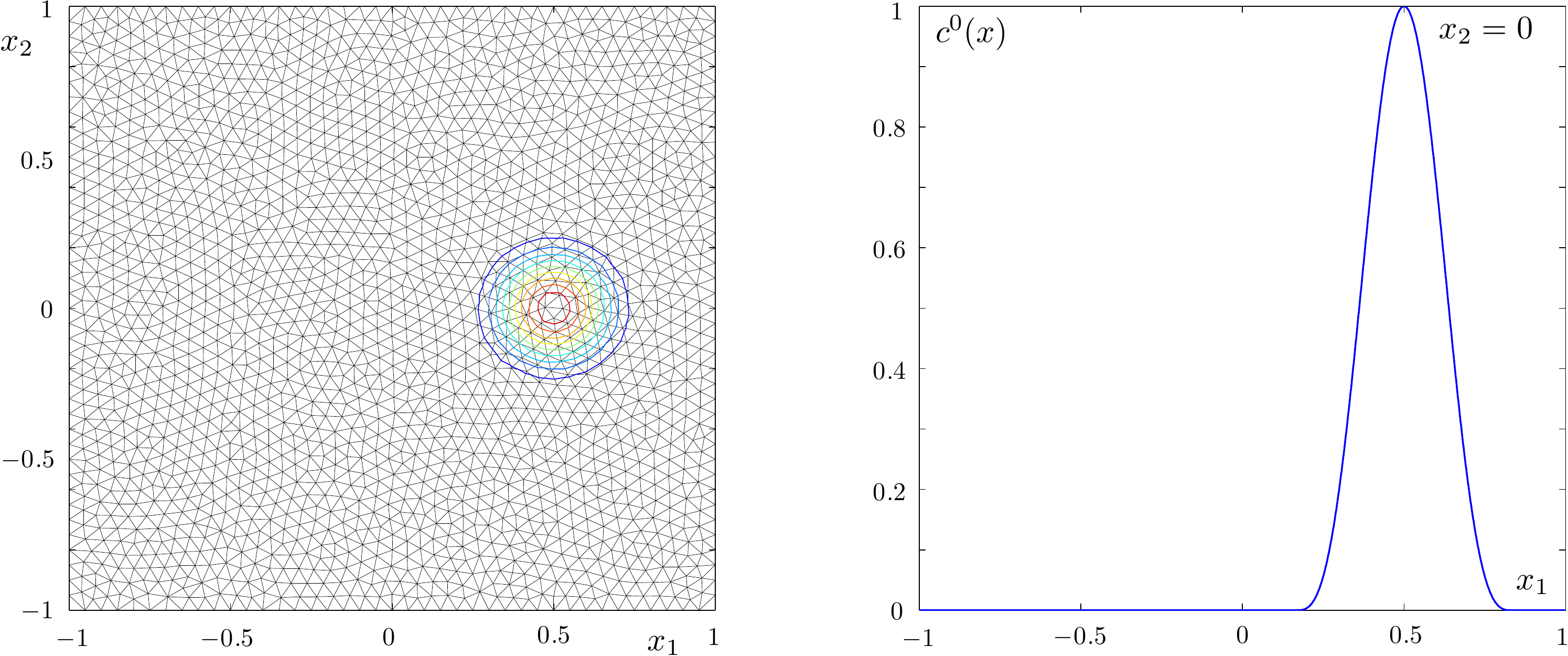}}
\end{center}
\caption{Initial condition of the rotating hump problem in a mesh with size $%
h=0.05$}
\label{figure:1}
\end{figure}
The purpose of this test is to see how the error behavior fits Theorem \ref%
{teorema1}. To do so, we shall mainly focus on the error as a function of
the parameter $\Delta t$. Since this theorem is valid under the assumption
that the integrals,
\begin{equation}
\int_{D}c_{h}^{n-1}\circ X^{n,n-1}(x)\phi _{i}(x)dx,  \label{gal_proj}
\end{equation}%
are calculated exactly, then we carry our goal out by using symmetric Gauss
quadrature rules of different orders of accuracy to evaluate such integrals;
in doing so, we assess the influence of the order of the quadrature rule on
the accuracy and stability of the numerical solution.

Figure 2 shows the $L^{2}$-norm of the error as a function of the time step $%
\Delta t$ in two meshes with $h=0.05$ and $h=0.025$ respectively. The errors
are calculated after one revolution, $T=1$, of the hump using quadrature
rules for the Galerkin projection (\ref{gal_proj}) of $7$, $16$, $25$, and $%
42$ points which are exact for polynomials of degree $5$, $8$, $10$, and $14$
respectively, see \cite{du}. Broken lines correspond to the error function
of linear polynomials ($m=1$), and full lines to the error function of
quadratic polynomials ($m=2$). By inspection, we notice the following items:
(a) for quadrature rules of high order, i.e., quadrature rules of 16, 25,
and 42 points, there is a value $\Delta t_{s}$, such that for $\Delta
t>\Delta t_{s}$, the error grows with a rate tending toward $O(\Delta
t^{-1/2})$ as $\Delta t$ decreases; the error tendency of the most accurate
rule of 42 points is closer to $O(\Delta t^{-1/2})$ than the error tendency
of the other two rules. On the other hand, for $0<\Delta t\leq \Delta t_{s}$%
, the error remains almost constant and independent of $\Delta t$. (b) For
quadrature rules that are not sufficiently accurate, i.e., the quadrature
rule of 7 points, there is a value $\Delta t_{ins}\gg \Delta t_{s}$ at which
the error starts growing very fast as $\Delta t$ decreases until it reaches
a maximum or eventually the numerical solution may become extremely large at
$\Delta t^{\ast }$. For $\Delta t_{s}<\Delta t<\Delta t^{\ast }$ the error
decreases and when $0<\Delta t\leq \Delta t_{s}$ the error remains constant.
This strange behavior of the solution for the quadrature rule of 7 points,
which illustrates the dependence of the stability of the LG method upon the
order of the quadrature rule, is a well known feature reported by many
authors, see for instance \cite{MPS}; in our tests, we note that the
instability with quadratic polynomials sends the numerical solution to
infinity in an interval of values of $\Delta t$, whereas for linear
polynomials the numerical solution, though useless, remains bounded.

\begin{figure}[th!]
\begin{center}
\scalebox{0.5}{\includegraphics{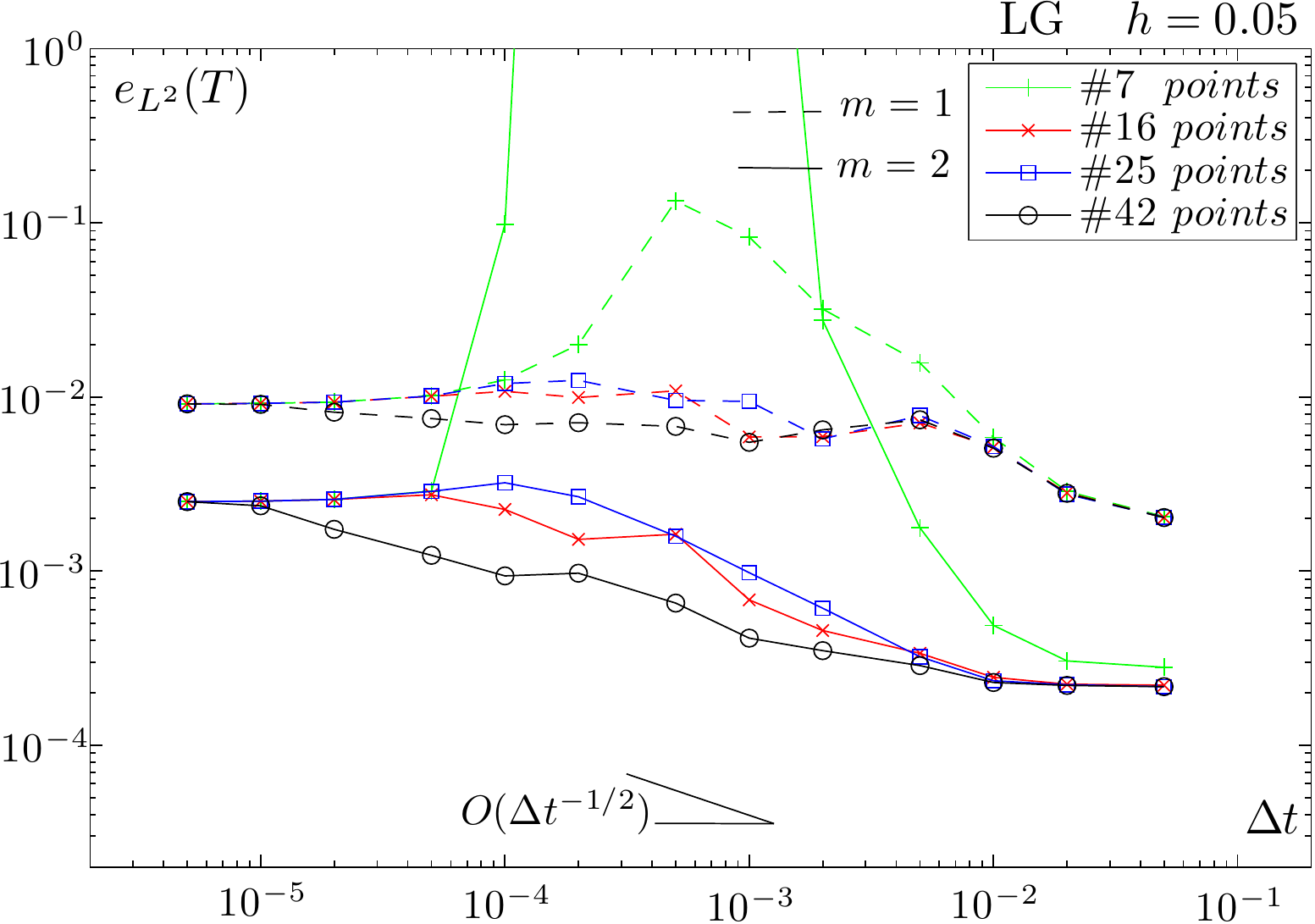}}\ \ \ %
\scalebox{0.5}{\includegraphics{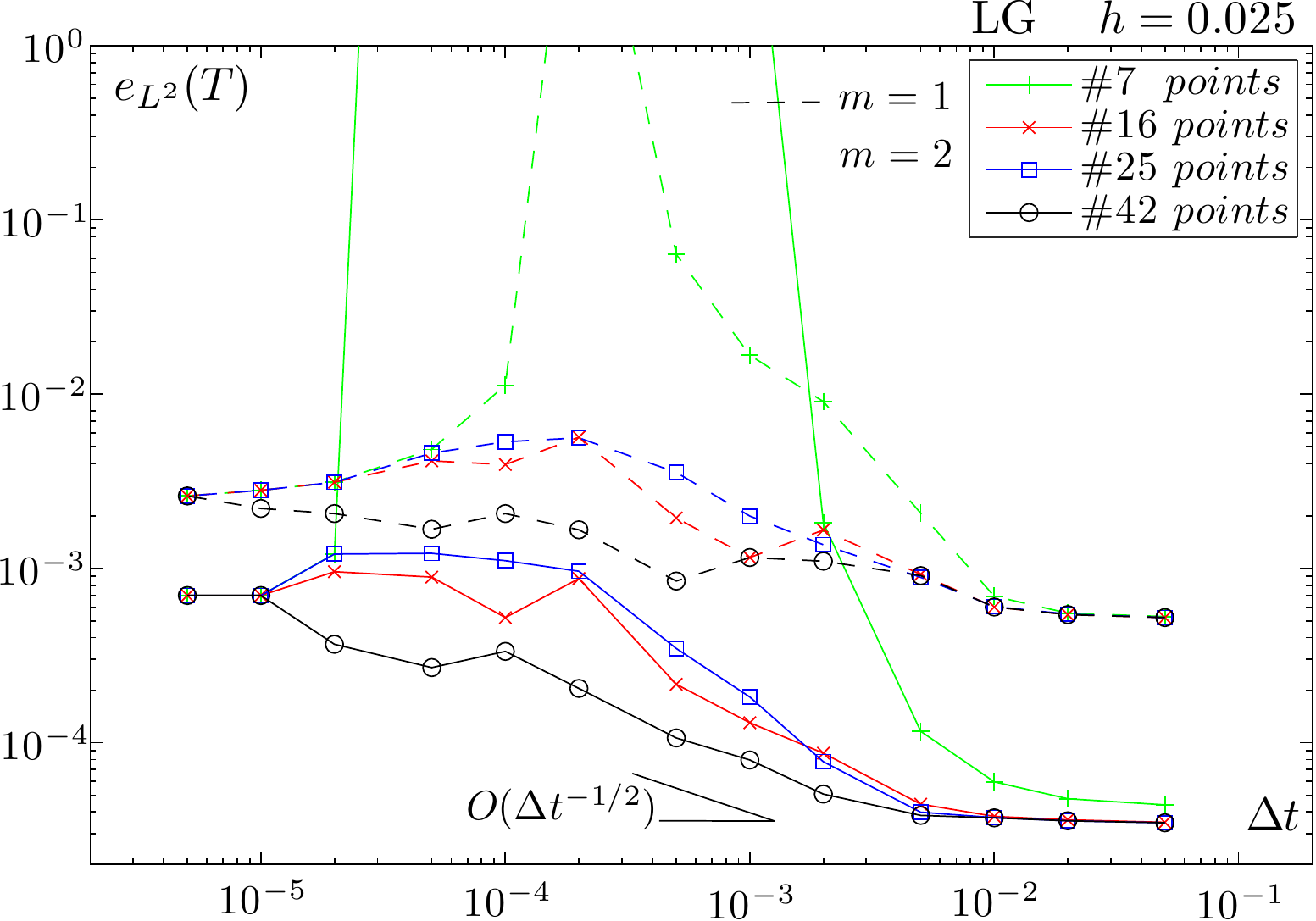}}
\end{center}
\caption{$L^2$-norm of the error of the conventional LG method in the
rotating hump problem, for linear $m=1$ and quadratic $m=2$ finite elements
in two different meshes ($h=0.05$ and $h=0.025$).}
\label{figure:2}
\end{figure}
Other relevant results displayed in Figure 2 are the following: (c) provided
that the integrals (\ref{gal_proj}) are evaluated with enough accuracy, the
numerical solutions are stable either for large or very small values of $%
\Delta t$, and as Theorem \ref{teorema1} says, the error is $O(h^{m+1}/{%
\Delta t^{1/2}})$ in the first case and $O(h^{m})$ in the second one, with
the particularity that in both cases the error does not depend very much
upon the order of the quadrature rule used to calculate (\ref{gal_proj}) as
long as such a rule is exact for polynomials of degree $>2(m+1)$. (d) The
error does not grow monotonically, though we notice that the higher the
order of the quadrature rule the smoother the growth of the error; however,
we can not explain why the rule of order $8$ ($16$ points) gives for some
values of $\Delta t$ smaller errors than the rule of order $10$ ($25$
points).

\section{The LPS-LG method}

To formulate the local projection stabilized Lagrange-Galerkin (LPS-LG)
method we introduce additional concepts. Besides the partition $D_{h}$, we
consider another quasi-uniform regular partition $\mathcal{M}_{h}$ on $D$
the elements of which are termed macro-elements. Each macro-element $M$ is
decomposed into one or more elements $K$ of the partition $D_{h}$ (the case $%
\mathcal{M}_{h}=D_{h}$ is allowed giving place to the so-called one-level
LPS approach). We assume that there exist positive constants $\gamma _{1}$
and $\gamma _{2}$ such that for all $K\subset D_{h}$ and $M\subset \mathcal{M%
}_{h}$, $\gamma _{1}h_{M}\leq h_{K}\leq \gamma _{2}h_{M}$. Next, we consider
a discontinuous finite element space $G_{h}$ associated with $\mathcal{M}%
_{h} $ and set $G_{h}(M):=\left\{ q_{h}\mid _{M}:q_{h}\in G_{h}\right\} $.
For each $M$, we use the local $L^{2}$-projector $\pi
_{M}:L^{2}(M)\rightarrow G_{h}(M)$ to define the fluctuation operator $%
\kappa _{M}:=id-\pi _{M}$, where $id:=L^{2}(M)\rightarrow L^{2}(M)$ is the
identity operator. In addition to the approximation properties (\ref{eq4a})-(%
\ref{eq4e}), we make the following assumptions.

\textbf{Assumption LPS1} \textit{Let }$s\in (0,\ldots ,m-1)$ \textit{be the
degree of the polynomials of the space }$G_{h}$, \textit{the fluctuation
operator} $\kappa _{M}$ \textit{satisfies the approximation property}%
\begin{equation}
\left\Vert \kappa _{M}w\right\Vert _{L^{2}(M)}\leq Ch_{M}^{l}\left\Vert
w\right\Vert _{H^{l}(M)},\ \forall w\in H^{l}(M),\ 0\leq l\leq s+1.
\label{p3}
\end{equation}%
Let $P_{s}(M)$\textit{\ }be the set of polynomials of degree at most $s$
defined in $M$, then a sufficient condition for the assumption \textbf{LPS1}
to hold is $P_{s}(M)\subset G_{h}(M)$. We set $W_{h}(M):=\left\{ w_{h}\mid
_{M}:w_{h}\in W_{h},\ w_{h}=0\text{ \textrm{on\ }}D\backslash M\right\} $.

\textbf{Assumption LPS2 }\textit{There is an interpolation operator }$%
j_{h}:H^{1}(D)\rightarrow W_{h}$,\textit{\ such that for all } $(w,q_{h})\in
H^{1}(D)\times G_{h}$,%
\begin{equation}
\left( w-j_{h}w,q_{h}\right) =0,  \label{p4}
\end{equation}%
\textit{and for all} $w\in H^{l}(D)$, \textit{with }$1\leq l\leq m+1$
\textit{and } $M\in \mathcal{M}_{h}$,
\begin{equation}
\left\Vert w-j_{h}w\right\Vert _{L^{2}(M)}+h_{M}\left\Vert \nabla
(w-j_{h}w)\right\Vert _{L^{2}(M)}\leq Ch_{M}^{l}\left\Vert w\right\Vert
_{H^{l}(\Lambda (M))},  \label{p5}
\end{equation}%
\textit{where }$\Lambda (M)$\textit{\ denotes a neighborhood of} $M$.

The existence of $j_{h}$ has been proven in Part III Chapter 3 of \cite{RST}
for spaces $G_{h}$ and $W_{h}$ that satisfy the following inf-sup condition :%
\begin{equation*}
\inf_{q_{h}\in G_{h}(M)}\sup_{w_{h}\in W_{h}(M)}\frac{\left(
q_{h},w_{h}\right) _{M}}{\left\Vert q_{h}\right\Vert _{L^{2}(M)}\left\Vert
w_{h}\right\Vert _{L^{2}(M)}}\geq \beta >0,
\end{equation*}%
where $\beta $ is a constant independent of $h$. For simplicial meshes the
spaces $(W_{h},G_{h})$ are the following (see, \cite{RST} for details):

let
\begin{equation*}
\begin{array}{c}
P_{m,h}^{disc}:=\{v_{h}\in L^{2}(D):v\mid _{K}=\widehat{v}\circ
F_{K}^{-1}\in \widehat{P}_{m}(\widehat{K})\ \ \forall K\in D_{h}\}\ \
\mathrm{and} \\
\\
P_{m,2h}^{disc}:=\{v_{h}\in L^{2}(D):v\mid _{M}=\widehat{v}\circ
F_{M}^{-1}\in \widehat{P}_{m}(\widehat{M})\ \ \forall M\in \mathcal{M}_{h}\}%
\text{,}%
\end{array}%
\end{equation*}%
where $F_{M}:\widehat{M}\rightarrow M\in \mathcal{M}_{h}$ is the bijective
transformation and $\widehat{M}$ is the reference element for the partition $%
\mathcal{M}_{h}$. The continuous finite element space $P_{m,h}$ is defined
as $P_{m,h}:=P_{m,h}^{disc}\cap H^{1}(D)$. For the one-level approach:%
\begin{equation}
W_{h}=P_{m,h}^{+}:=P_{m,h}+\mathrm{span}_{K\in M_{h}}\{\Phi _{K}\cdot
P_{m-1,h}(K)\},\ \mathrm{and}\ G_{h}=P_{m-1,h}^{disc},  \label{p6}
\end{equation}%
here, $\Phi _{K}$ denotes the mapped bubble function that vanishes on the
boundary $\partial K$ of the element. For the two-level approach (the
elements $K\in T_{h}$ are obtained from the elements $M\in \mathcal{M}_{h}$
by means of a refinement criterium, see for instance \cite{BB} and \cite{GT}%
):%
\begin{equation}
W_{h}=P_{m,h}\text{,}\ \mathrm{and}\text{ }G_{h}=P_{m-1,2h}^{disc}\text{.}
\end{equation}
Figure \ref{figure:lps1} illustrates these approaches for $d=2$ and
simplicial meshes. 
\begin{figure}[th!]
\begin{center}
\scalebox{0.5}{\includegraphics{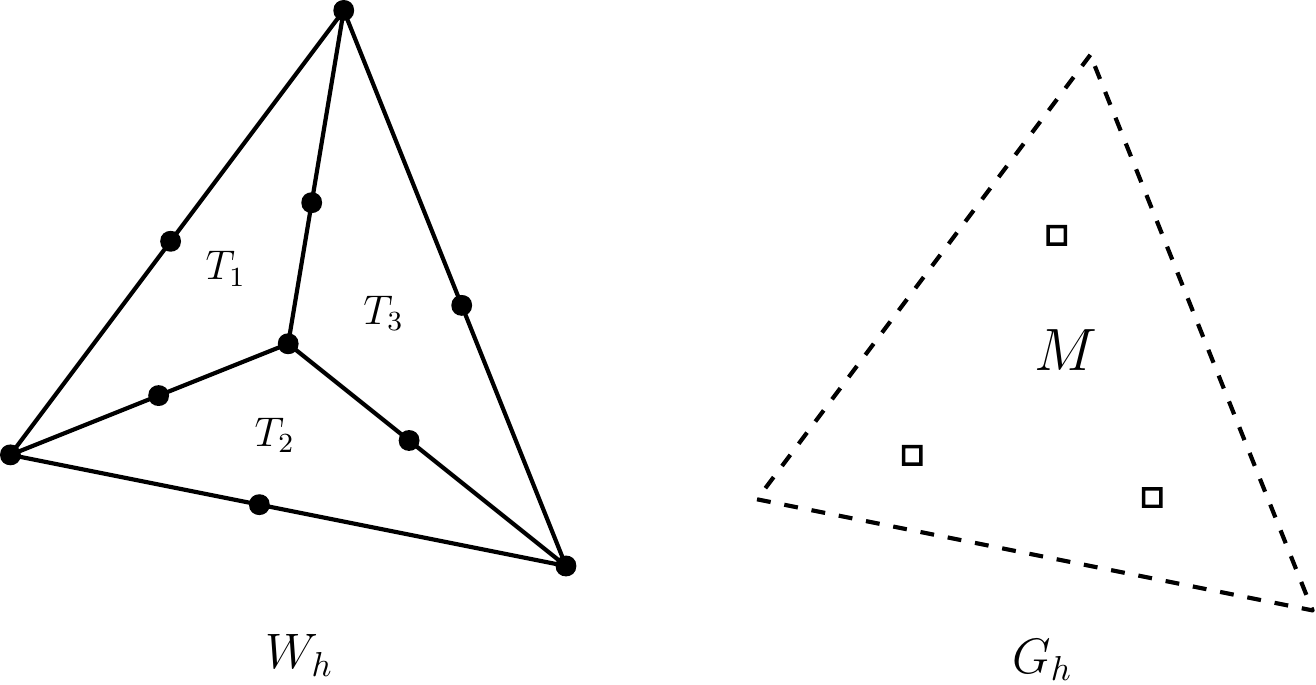}}\ \ \ \scalebox{0.5}{%
\includegraphics{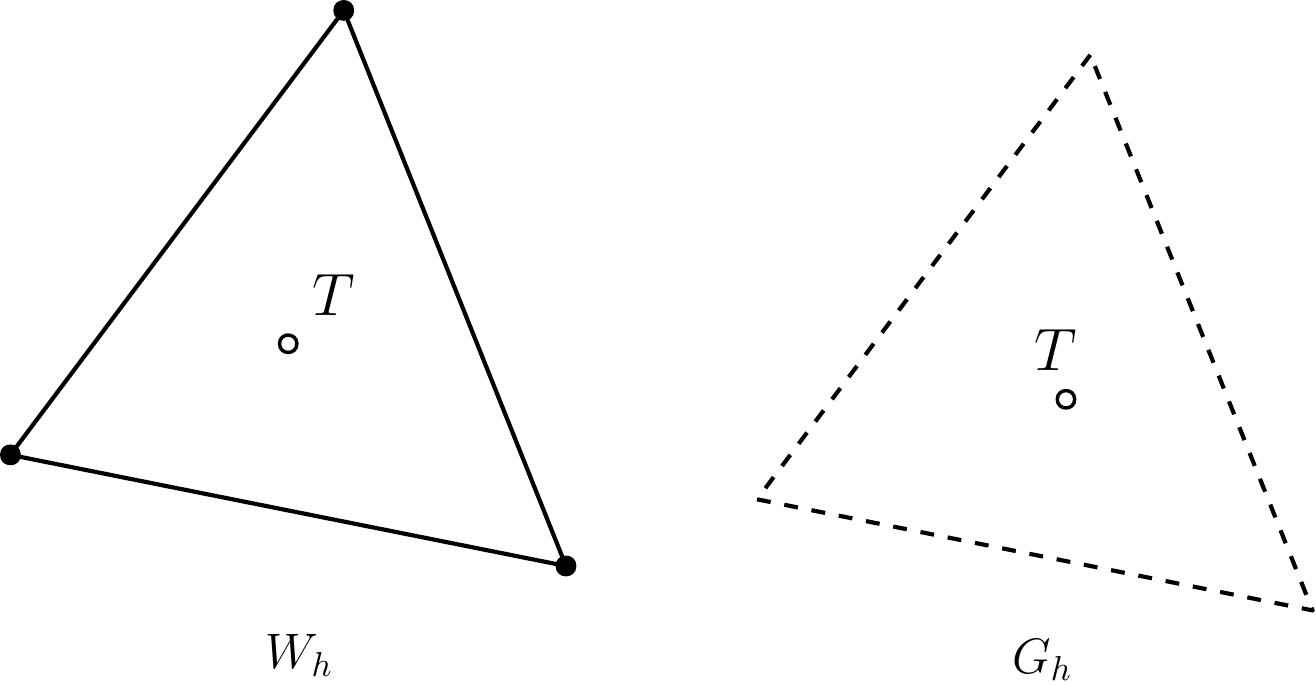}}
\end{center}
\caption{Approximation and projection spaces. Left panel, the two-level
approach with $m=2$; right panel, the one-level approach with $m=1$.}
\label{figure:lps1}
\end{figure}

The LPS Lagrange-Galerkin method calculates $c_{h}^{n}\in V_{h}$ as solution
of the equation%
\begin{equation}
(c_{h}^{n}-c_{h}^{\ast n-1},v_{h})+\Delta tS_{h}(c_{h}^{n},v_{h})=0\ \
\forall v_{h}\in V_{h},  \label{lps1}
\end{equation}%
where $S_{h}(c_{h}^{n},v_{h})$ is the stabilization term given by the
expression
\begin{equation}
S_{h}(c_{h}^{n},v_{h})=\sum_{M}\tau _{M}(\kappa _{M}\nabla c_{h}^{n},\kappa
_{M}\nabla v_{h})_{M},  \label{lps2}
\end{equation}%
here, $(\kappa _{M}\nabla c_{h}^{n},\kappa _{M}\nabla
v_{h})_{M}:=\int_{M}\kappa _{M}\nabla c_{h}^{n}\cdot\kappa _{M}\nabla
v_{h}dx $ and $\mathbf{\tau }_{M}$ are element-wise constant coefficients
that depend on the diameter $h_{M}$ of the macro-elements, their optimal
values are determined by the error analysis.

\begin{remark}
For $v_{h}=c_{h}^{n}$ the term $S_{h}(c_{h}^{n},v_{h})$ can be written as a
diffusion term of the form%
\begin{equation*}
S_{h}(c_{h}^{n},c_{h}^{n})=\sum_{M}\tau _{M}\left\Vert \mathbf{\kappa }%
_{M}\nabla c_{h}^{n}\right\Vert _{\mathbf{L}^{2}(M)}^{2}:=\nu _{\text{add}%
}(c_{h}^{n})\left\Vert \nabla c_{h}^{n}\right\Vert ^{2},
\end{equation*}%
where%
\begin{equation*}
\nu _{\text{add}}(c_{h}^{n}):=\left\{
\begin{array}{l}
\dfrac{\sum_{M}\tau _{M}\left\Vert \mathbf{\kappa }_{M}\nabla
c_{h}^{n}\right\Vert _{\mathbf{L}^{2}(M)}^{2}}{\left\Vert \nabla
c_{h}^{n}\right\Vert ^{2}}\ \mathrm{when}\text{ }\left\Vert \nabla
c_{h}^{n}\right\Vert \neq 0\text{,} \\
\\
0\text{\ \ \ \ otherwise.}%
\end{array}%
\right. .
\end{equation*}
\end{remark}

\subsection{Analysis of the LPS-LG method}

We prove the stability of the LPS-LG method in the mesh dependent norm%
\begin{equation}
\left\vert \left\vert \left\vert v^{n}\right\vert \right\vert \right\vert
=\left( \left\vert \left\vert v^{n}\right\vert \right\vert ^{2}+\Delta
t\sum_{j=1}^{n}S_{h}(v^{j},v^{j})\right) ^{1/2},  \label{lps5}
\end{equation}%
where $v^{j}\in H_{0}^{1}(D)$ $(j=1,..,n)$, $n$ being a positive integer. We
have the following result.

\begin{lemma}
\label{stablps} For all $N\geq 1$ it holds
\begin{equation}
\left\vert \left\vert \left\vert c_{h}^{N}\right\vert \right\vert
\right\vert ^{2}+\sum_{i=1}^{N}\left\Vert c_{h}^{i}-c_{h}^{\ast
i-1}\right\Vert ^{2}\leq \left\Vert c_{h}^{o}\right\Vert ^{2}.  \label{lp6}
\end{equation}
\end{lemma}

\begin{proof}
Let $v_{h}=c_{h}^{n}$ in (\ref{lps1}), then it follows that%
\begin{equation*}
\left\Vert c_{h}^{n}\right\Vert ^{2}+\left\Vert c_{h}^{n}-c_{h}^{\ast
n-1}\right\Vert ^{2}-\left\Vert c_{h}^{\ast n-1}\right\Vert ^{2}+2\Delta
tS_{h}(c_{h}^{n},c_{h}^{n})=0.
\end{equation*}%
Noting that by virtue of (\ref{eq6.0}), $\left\Vert c_{h}^{\ast
n-1}\right\Vert ^{2}=\left\Vert c_{h}^{n-1}\right\Vert ^{2}$, then summing
from $n=1$ up to $n=N\geq 1$ yields%
\begin{equation*}
\left\Vert c_{h}^{N}\right\Vert ^{2}+\sum_{n=1}^{N}\left\Vert
c_{h}^{n}-c_{h}^{\ast n-1}\right\Vert ^{2}+2\Delta t\sum_{n=1}^{N}
S_{h}(c_{h}^{n},c_{h}^{n})=\left\Vert c_{h}^{0}\right\Vert ^{2}.
\end{equation*}
Hence, (\ref{lp6}) follows.
\end{proof}

Next, we perform the error analysis. To do so, we again decompose the error
function $e^{n}:=c^{n}-c_{h}^{n}$ as%
\begin{equation}
e^{n}=(c^{n}-P_{h}c^{n})+(P_{h}c^{n}-c_{h}^{n})\equiv \rho ^{n}+\theta
_{h}^{n}.  \label{lps7}
\end{equation}

First, we calculate an estimate for $\left\vert \left\vert \left\vert \rho
^{L}\right\vert \right\vert \right\vert $.

\begin{lemma}
\label{rholps} Let $c\in L^{\infty }(H^{m+1}(D)\cap H_{0}^{1}(D))$. Then,
for all $0\leq L\leq N$. it follows that there exists a constant $C$
independent of $h$, $\Delta t$ and $L$ such that%
\begin{equation}
\left\vert \left\vert \left\vert \rho ^{L}\right\vert \right\vert
\right\vert \leq C( h+\tau _{\max }^{1/2}) h^{m}\left\Vert c\right\Vert
_{L^{\infty }(H^{m+1}(D))},  \label{lps8}
\end{equation}%
where $\tau _{\max }=\max_{M\in \mathcal{M}_{h}}(\tau _{M})$.
\end{lemma}

\begin{proof}
We recall that%
\begin{equation*}
\left\vert \left\vert \left\vert \rho ^{L}\right\vert \right\vert
\right\vert ^{2}=\left\Vert \rho ^{L}\right\Vert ^{2}+\Delta
t\sum_{n=0}^{L}S_{h}(\rho ^{n},\rho ^{n}).
\end{equation*}%
So, by virtue of (\ref{eq4d}) it follows that for all $L$ there is a
constant independent of $h$, such that%
\begin{equation*}
\left\Vert \rho ^{L}\right\Vert \leq Ch^{m+1}\left\vert c\right\vert
_{L^{\infty }(H^{m+1}(D))}.
\end{equation*}%
Next, we estimate the term $S_{h}(\rho ^{n},\rho ^{n})$. Making use of the
triangle inequality, the contractiveness property of the local $L^{2}$%
-projector $\pi _{M}$, and (\ref{eq4d}) we obtain that%
\begin{equation}
\begin{array}{r}
S_{h}(\rho ^{n},\rho ^{n})=\sum_{M}\tau _{M}\left\Vert \kappa _{M}\nabla
\rho ^{n}\right\Vert _{L^{2}(M)}^{2}\leq 4\sum_{M}\tau _{M}\left\Vert \nabla
\rho ^{n}\right\Vert _{L^{2}(M)}^{2} \\
\\
\leq C\tau _{\max }h^{2m}\left\Vert c\right\Vert _{L^{\infty }\left(
H^{m+1}(D)\right) }^{2}.%
\end{array}
\label{lps81}
\end{equation}%
Hence, collecting these two estimates the result (\ref{lps8}) follows.
\end{proof}

We are ready to establish the convergence of the LPS-LG method.

\begin{theorem}
\label{teorlps} Under the assumptions of Theorem \ref{teorema1}, there
exists a constant $C$ independent of $h$, $\Delta t$ and $L$, such that for
all $L $, $0\leq L\leq N$,
\begin{equation}
\max_{0\leq L\leq N}\left\vert \left\vert \left\vert e^{L}\right\vert
\right\vert \right\vert \leq \left\vert \left\vert \left\vert
e^{0}\right\vert \right\vert \right\vert +C_{3}\left( \tau _{\max
}^{1/2}(h^{m}+h^{s+1})+h^{m+1}+\min \left( 1,\frac{\Delta t^{1/2}\left\Vert
\mathbf{u}\right\Vert _{L^{\infty }(L^{2}(D))}}{h}\right) \left( \frac{%
h^{m+1}}{\Delta t^{1/2}}\right) \right) .  \label{lps82}
\end{equation}
\end{theorem}

\begin{proof}
From (\ref{eq4}) with $t=t_{n}$ and $\tau =-\Delta t$ and (\ref{lps1}) we
obtain the error equation
\begin{equation}
(e^{n}-e^{\ast n-1},v_{h})-\Delta tS_{h}(c_{h}^{n},v_{h})=0.  \label{lpse0}
\end{equation}%
Noting that $c_{h}^{n}=c^{n}-e^{n}$, we recast this equation as%
\begin{equation*}
(e^{n}-e^{\ast n-1},v_{h})+\Delta tS_{h}(e^{n},v_{h})=\Delta
tS_{h}(c^{n},v_{h}).
\end{equation*}%
Next, setting $v_{h}=\theta _{h}^{n}=e^{n}-\rho ^{n}$ and observing that $%
S_{h}(a+b,c)=S_{h}(a,c)+S_{h}(b,c)$ this equation becomes%
\begin{equation}
\begin{array}{l}
(e^{n}-e^{\ast n-1},e^{n})+\Delta tS_{h}(e^{n},e^{n})=(e^{n}-e^{\ast
n-1},\rho ^{n}) \\
\\
+\Delta tS_{h}(e^{n},\rho ^{n})+\Delta tS_{h}(e^{n},c^{n})-\Delta
tS_{h}(c^{n},\rho ^{n})\equiv \displaystyle\sum_{i=1}^{4}T_{i}%
\end{array}
\label{lpse1}
\end{equation}%
We estimate the terms $T_{i}$ of (\ref{lpse1}). Applying Cauchy-Schwarz
inequality we have that%
\begin{equation*}
\left\vert T_{1}\right\vert =\left\vert (e^{n}-e^{\ast n-1},\rho
^{n})\right\vert \leq \frac{1}{4}\left\Vert e^{n}-e^{\ast n-1}\right\Vert
^{2}+\left\Vert \rho ^{n}\right\Vert ^{2}.
\end{equation*}%
To estimate $T_{2}$ we apply again Cauchy-Schwarz inequality and obtain%
\begin{equation*}
\begin{array}{r}
\left\vert T_{2}\right\vert =\Delta t\left\vert S_{h}(e^{n},\rho
^{n})\right\vert \leq \Delta t\left( S_{h}(e^{n},e^{n})\right) ^{1/2}\left(
S_{h}(\rho ^{n},\rho ^{n})\right) ^{1/2} \\
\\
\leq (\delta /2)\Delta tS_{h}(e^{n},e^{n})+(2\delta )^{-1}\Delta tS_{h}(\rho
^{n},\rho ^{n}),\ \ 0<\delta <1.%
\end{array}%
\end{equation*}%
Similarly, we have that%
\begin{equation*}
\left\vert T_{3}\right\vert \leq (\delta /2)\Delta
tS_{h}(e^{n},e^{n})+(2\delta )^{-1}\Delta tS_{h}(c^{n},c^{n}),
\end{equation*}%
and%
\begin{equation*}
\left\vert T_{4}\right\vert \leq \frac{\Delta t}{2}S_{h}(c^{n},c^{n})+\frac{%
\Delta t}{2}S_{h}(\rho ^{n},\rho ^{n}).
\end{equation*}%
Substituting these estimates in (\ref{lpse1}) with $\delta =1/2$ and noting
that%
\begin{equation*}
2(e^{n}-e^{\ast n-1},e^{n})=\left\Vert e^{n}\right\Vert ^{2}+\left\Vert
e^{n}-e^{\ast n-1}\right\Vert ^{2}-\left\Vert e^{n-1}\right\Vert ^{2},
\end{equation*}%
we obtain that%
\begin{equation*}
\begin{array}{r}
\left\Vert e^{n}\right\Vert ^{2}+\frac{1}{2}\left\Vert e^{n}-e^{\ast
n-1}\right\Vert ^{2}-\left\Vert e^{n-1}\right\Vert ^{2}+\Delta
tS_{h}(e^{n},e^{n}) \\
\\
\leq 2\left\Vert \rho ^{n}\right\Vert ^{2}+3\Delta t\left(
S_{h}(c^{n},c^{n})+S_{h}(\rho ^{n},\rho ^{n})\right) .%
\end{array}%
\end{equation*}%
Summing both terms of this inequality from $n=1$ up to $n=N$ yields%
\begin{equation*}
\begin{array}{r}
\left\Vert e^{N}\right\Vert ^{2}+\frac{1}{2}\sum_{n=1}^{N}\left\Vert
e^{n}-e^{\ast n-1}\right\Vert ^{2}-\left\Vert e^{0}\right\Vert ^{2}+\Delta
t\sum_{n=1}^{N}S_{h}(e^{n},e^{n}) \\
\\
\leq \sum_{n=1}^{N}\left\Vert \rho ^{n}\right\Vert ^{2}+5\Delta
t\sum_{n=1}^{N}\left( S_{h}(c^{n},c^{n})+S_{h}(\rho ^{n},\rho ^{n})\right) .%
\end{array}%
\end{equation*}%
Since for any non negative integer $n$,%
\begin{equation*}
\left\Vert \rho ^{n}\right\Vert ^{2}\leq Ch^{2(m+1)}\left\vert c\right\vert
_{L^{\infty }(H^{m+1}(D))}^{2},
\end{equation*}%
by virtue of assumption \textbf{LPS1}

\begin{equation}
S_{h}(c^{n},c^{n})\leq C\tau _{\max }h^{2(s+1)}\left\vert c\right\vert
_{L^{\infty }(H^{s+2}(D))}^{2},  \label{lpse1.01}
\end{equation}%
and observing that $\left\Vert \kappa _{M}\nabla \rho ^{n}\right\Vert \leq
2\left\Vert \nabla \rho ^{n}\right\Vert $ because $\pi _{M}$ is contractive,
then (see (\ref{lps81}))%
\begin{equation}
S_{h}(\rho ^{n},\rho ^{n})\leq C\tau _{\max }h^{2m}\left\vert c\right\vert
_{L^{\infty }(H^{m+1}(D))}^{2},  \label{lpse1.02}
\end{equation}%
we have that%
\begin{equation*}
\begin{array}{r}
\left\Vert e^{N}\right\Vert ^{2}+\frac{1}{2}\sum_{n=1}^{N}\left\Vert
e^{n}-e^{\ast n-1}\right\Vert ^{2}-\left\Vert e^{0}\right\Vert ^{2}+\Delta
t\sum_{n=1}^{N}S_{h}(e^{n},e^{n}) \\
\\
\leq C\left( \tau _{\max }(h^{2m}+h^{2(s+1)})+\frac{h^{2(m+1)}}{\Delta t}%
\right) ,%
\end{array}%
\end{equation*}%
or equivalently%
\begin{equation}
\left\vert \left\vert \left\vert e^{N}\right\vert \right\vert \right\vert
^{2}+\frac{1}{2}\sum_{n=1}^{N}\left\Vert e^{n}-e^{\ast n-1}\right\Vert
^{2}\leq \left\vert \left\Vert e^{0}\right\Vert \right\vert ^{2}+C\left(
\tau _{\max }(h^{2m}+h^{2\left( s+1\right) })+\frac{h^{2\left( m+1\right) }}{%
\Delta t}\right) .  \label{lpse2.00}
\end{equation}%
Hence, it follows that%
\begin{equation}
\max_{0\leq L\leq N}\left\vert \left\vert \left\vert e^{L}\right\vert
\right\vert \right\vert \leq \left\vert \left\vert \left\vert
e^{0}\right\vert \right\vert \right\vert +C\left( \tau _{\max }^{1/2}\left(
h^{m}+h^{s+1}\right) +\frac{h^{m+1}}{\Delta t^{1/2}}\right) .  \label{lpse2}
\end{equation}%
This estimate of the error depends on $\Delta t^{-1/2}$ so that, for any
fixed $h$, is invalid when $\Delta t\rightarrow 0$ because in this case the
method does not converge. So, in order to get rid of the factor $\Delta
t^{-1/2}$ we consider the following approach. Starting with the error
equation (\ref{lpse0}) and setting%
\begin{equation*}
v_{h}=\theta _{h}^{n},\ e^{n}=\rho ^{n}+\theta _{h}^{n},\ e^{\ast n-1}=\rho
^{\ast n-1}+\theta _{h}^{\ast n-1}\ \mathrm{and\ }c_{h}^{n}=c^{n}-(\rho
^{n}+\theta _{h}^{n}),
\end{equation*}%
we get%
\begin{equation*}
\begin{array}{r}
\left( \theta _{h}^{n}-\theta _{h}^{\ast n-1},\theta _{h}^{n}\right) +\Delta
tS_{h}(\theta _{h}^{n},\theta _{h}^{n})=-(\rho ^{n}-\rho ^{\ast n-1},\theta
_{h}^{n}) \\
\\
+\Delta t\left( S_{h}(\rho ^{n},\theta _{h}^{n})-S_{h}(c^{n},\theta
_{h}^{n})\right) .%
\end{array}%
\end{equation*}%
Now, noticing that $\rho ^{n}-\rho ^{\ast n-1}=\rho ^{n}-\rho ^{n-1}-\left(
\rho ^{\ast n-1}-\rho ^{n-1}\right) $ and $\left( \rho ^{n}-\rho
^{n-1},\theta _{h}^{n}\right) =0$, we can write the above equation as%
\begin{equation}
\begin{array}{l}
\frac{1}{2}\left\Vert \theta _{h}^{n}\right\Vert ^{2}+\frac{1}{4}\left\Vert
\theta _{h}^{n}-\theta _{h}^{\ast n-1}\right\Vert ^{2}-\frac{1}{2}\left\Vert
\theta _{h}^{n-1}\right\Vert ^{2}+\Delta tS_{h}(\theta _{h}^{n},\theta
_{h}^{n}) \\
\\
\leq -(\rho ^{n-1}-\rho ^{\ast n-1},\theta _{h}^{n})+\Delta t\left(
S_{h}(\rho ^{n},\theta _{h}^{n})-S_{h}(c^{n},\theta _{h}^{n})\right) \equiv
\sum_{i=1}^{3}S_{i}.%
\end{array}
\label{lpse3}
\end{equation}%
We bound the terms $S_{i}$ on the right hand side of (\ref{lpse3}). Thus, by
the Cauchy-Schwarz inequality we have that%
\begin{equation*}
\left\vert S_{1}\right\vert \leq \frac{1}{\Delta t}\left\Vert \rho
^{n-1}-\rho ^{\ast n-1}\right\Vert ^{2}+\frac{\Delta t}{4}\left\Vert \theta
_{h}^{n}\right\Vert ^{2};
\end{equation*}%
since, see (\ref{roro*}),
\begin{equation*}
\frac{1}{\Delta t}\left\Vert \rho ^{n-1}-\rho ^{\ast n-1}\right\Vert
^{2}\leq \Delta tC\left( \frac{\Delta t^{1/2}\left\Vert \mathbf{u}%
\right\Vert _{L^{\infty }(L^{\theta }(D))}}{h}\right) ^{2}\left( \frac{%
h^{m+1}}{\Delta t^{1/2}}\right) ^{2},
\end{equation*}%
then%
\begin{equation}
\left\vert S_{1}\right\vert \leq \Delta tC\left( \frac{\Delta
t^{1/2}\left\Vert \mathbf{u}\right\Vert _{L^{\infty }(L^{\theta }(D))}}{h}%
\right) ^{2}\left( \frac{h^{m+1}}{\Delta t^{1/2}}\right) ^{2}+\frac{\Delta t%
}{4}\left\Vert \theta _{h}^{n}\right\Vert ^{2}.  \label{lpse4}
\end{equation}%
To bound the terms $S_{2}$ and $S_{3}$ we use the same technique as for the
terms $T_{1}$ and $T_{2}$ above and obtain%
\begin{equation*}
\left\vert S_{2}\right\vert =\Delta t\left\vert S_{h}(\rho ^{n},\theta
_{h}^{n})\right\vert \leq (\delta /2)\Delta tS_{h}(\theta _{h}^{n},\theta
_{h}^{n})+(2\delta )^{-1}\Delta tS_{h}(\rho ^{n},\rho ^{n}),
\end{equation*}%
and%
\begin{equation*}
\left\vert S_{3}\right\vert =\Delta t\left\vert S_{h}(c^{n},\theta
_{h}^{n})\right\vert \leq (\delta /2)\Delta tS_{h}(\theta _{h}^{n},\theta
_{h}^{n})+(2\delta )^{-1}\Delta tS_{h}(c^{n},c^{n}).
\end{equation*}%
Setting $\delta =1/2$ and substituting these bounds in (\ref{lpse3}) yields%
\begin{equation*}
\begin{array}{l}
\left\Vert \theta _{h}^{n}\right\Vert ^{2}+\frac{1}{2}\left\Vert \theta
_{h}^{n}-\theta _{h}^{\ast n-1}\right\Vert ^{2}-\left\Vert \theta
_{h}^{n-1}\right\Vert ^{2}+\Delta tS_{h}(\theta _{h}^{n},\theta _{h}^{n}) \\
\\
\leq \Delta tC\displaystyle\left( \frac{\Delta t^{1/2}\left\Vert \mathbf{u}%
\right\Vert _{L^{\infty }(L^{\theta }(D))}}{h}\right) ^{2}\left( \frac{h^{m}%
}{\Delta t^{1/2}}\right) ^{2}+5\Delta t\left( S_{h}(\rho ^{n},\rho
^{n})\right)  \\
\\
+5\Delta t\left( S_{h}(c^{n},c^{n})\right) +\displaystyle\frac{\Delta t}{2}%
\left\Vert \theta _{h}^{n}\right\Vert ^{2}.%
\end{array}%
\end{equation*}%
Or equivalently, using (\ref{lpse1.01}) and (\ref{lpse1.02}),%
\begin{equation*}
\begin{array}{l}
\left\Vert \theta _{h}^{n}\right\Vert ^{2}+\frac{1}{2}\left\Vert \theta
_{h}^{n}-\theta _{h}^{\ast n-1}\right\Vert ^{2}-\left\Vert \theta
_{h}^{n-1}\right\Vert ^{2}+\Delta tS_{h}(\theta _{h}^{n},\theta _{h}^{n}) \\
\\
\leq \Delta tC\left( \tau _{\max }(h^{2m}+h^{2\left( s+1\right) })+\left(
\frac{\Delta t^{1/2}\left\Vert \mathbf{u}\right\Vert _{L^{\infty }(L^{\theta
}(D))}}{h}\right) ^{2}\left( \frac{h^{m}}{\Delta t^{1/2}}\right) ^{2}\right)
+\displaystyle\frac{\Delta t}{2}\left\Vert \theta _{h}^{n}\right\Vert ^{2}.%
\end{array}%
\end{equation*}%
Summing both sides of this inequality from $n=1$ up to $n=N$ and applying
Gronwall inequality we obtain that%
\begin{equation}
\begin{array}{c}
\left\vert \left\vert \left\vert \theta _{h}^{N}\right\vert \right\vert
\right\vert ^{2}+\frac{1}{2}\sum_{n=1}^{N}\left\Vert \theta _{h}^{n}-\theta
_{h}^{\ast n-1}\right\Vert ^{2}\leq \left\Vert \theta _{h}^{0}\right\Vert
^{2} \\
\\
+C\left( \tau _{\max }(h^{2m}+h^{2\left( s+1\right) })\displaystyle\left(
\frac{\Delta t^{1/2}\left\Vert \mathbf{u}\right\Vert _{L^{\infty }(L^{\theta
}(D))}}{h}\right) ^{2}\left( \frac{h^{m+1}}{\Delta t^{1/2}}\right)
^{2}\right) .%
\end{array}
\label{lpse5.0}
\end{equation}%
Hence,%
\begin{equation}
\left\vert \left\vert \left\vert \theta _{h}^{N}\right\vert \right\vert
\right\vert \leq \left\vert \left\vert \left\vert \theta _{h}^{0}\right\vert
\right\vert \right\vert +C\left( \tau _{\max }^{1/2}(h^{m}+h^{s+1})+\left(
\frac{\Delta t^{1/2}\left\Vert \mathbf{u}\right\Vert _{L^{\infty }(L^{\theta
}(D))}}{h}\right) \left( \frac{h^{m+1}}{\Delta t^{1/2}}\right) \right) .
\label{lpse5.1}
\end{equation}%
Now, noting that $\left\vert \left\vert \left\vert e^{N}\right\vert
\right\vert \right\vert \leq \left\vert \left\vert \left\vert \theta
_{h}^{N}\right\vert \right\vert \right\vert +\left\vert \left\vert
\left\vert \rho ^{N}\right\vert \right\vert \right\vert $ and%
\begin{equation*}
\left\vert \left\vert \left\vert \rho ^{N}\right\vert \right\vert
\right\vert \leq C(h^{m+1}+\tau _{\max }^{1/2}h^{m}),
\end{equation*}%
it follows from (\ref{lpse5.1}) that%
\begin{equation}
\left\vert \left\vert \left\vert e^{N}\right\vert \right\vert \right\vert
\leq \left\vert \left\vert \left\vert e^{0}\right\vert \right\vert
\right\vert +C\left( \tau _{\max }^{1/2}(h^{m}+h^{s+1})+h^{m+1}+\left( \frac{%
\Delta t^{1/2}\left\Vert \mathbf{u}\right\Vert _{L^{\infty }(L^{2}(D))}}{h}%
\right) \left( \frac{h^{m+1}}{\Delta t^{1/2}}\right) \right) .
\label{lpse5.2}
\end{equation}%
Thus, since both estimates (\ref{lpse5.2}) and (\ref{lpse2}) hold, then we
can write that there exists a constant $C$ such that%
\begin{equation*}
\left\vert \left\vert \left\vert e^{N}\right\vert \right\vert \right\vert
\leq \left\vert \left\vert \left\vert e^{0}\right\vert \right\vert
\right\vert +C\left( \tau _{\max }^{1/2}(h^{m}+h^{s+1})+h^{m+1}+\min \left(
1,\frac{\Delta t^{1/2}\left\Vert \mathbf{u}\right\Vert _{L^{\infty
}(L^{2}(D))}}{h}\right) \left( \frac{h^{m+1}}{\Delta t^{1/2}}\right) \right)
.
\end{equation*}
\end{proof}

\bigskip

\subsection{Numerical tests with the LPS-LG method}

We run, under the same premises, the rotating hump problem defined in
Section 3.3, although the mesh is now composed of right triangles with legs
of length $h=10^{-2}$. We show in the upper panel of Figure \ref{figure:lps2}
the $L^{2}$-norm of the error as a function of $\Delta t$ for both the
two-level LPS-LG method and the conventional LG method for a mesh size $h=%
\sqrt{2}\times 10^{-2}$ in both cases. In these experiments we have
calculated the integrals (\ref{gal_proj}) with a quadrature rule of 12
points, which is exact for polynomials of degree 6. The spaces $W_{h}$ and $%
G_{h}$ of the LPS-LG method are those shown in Figure \ref{figure:lps1},
whereas the finite element space for the conventional LG method consists of
piecewise quadratic polynomials defined on each one of the 3 triangles that
compose the macro-element. We observe that the LPS-LG method is more stable
than the conventional LG, because the latter goes unstable whereas the
LPS-LG method remains stable when $\tau _{M}=\tau =h$ for all $M$, but it
becomes unstable, with an instability region along the $\Delta t$-axis
smaller than the one of the LG method, when $\tau _{M}=\tau =0.1h$ for all $M
$. The lower panel of the figure shows that by increasing the order of the
quadrature rule the LPS-LG method with $\tau _{M}=\tau =0.1h$ becomes
stable. 

\begin{figure}[th!]
\begin{center}
\scalebox{0.30}{\includegraphics{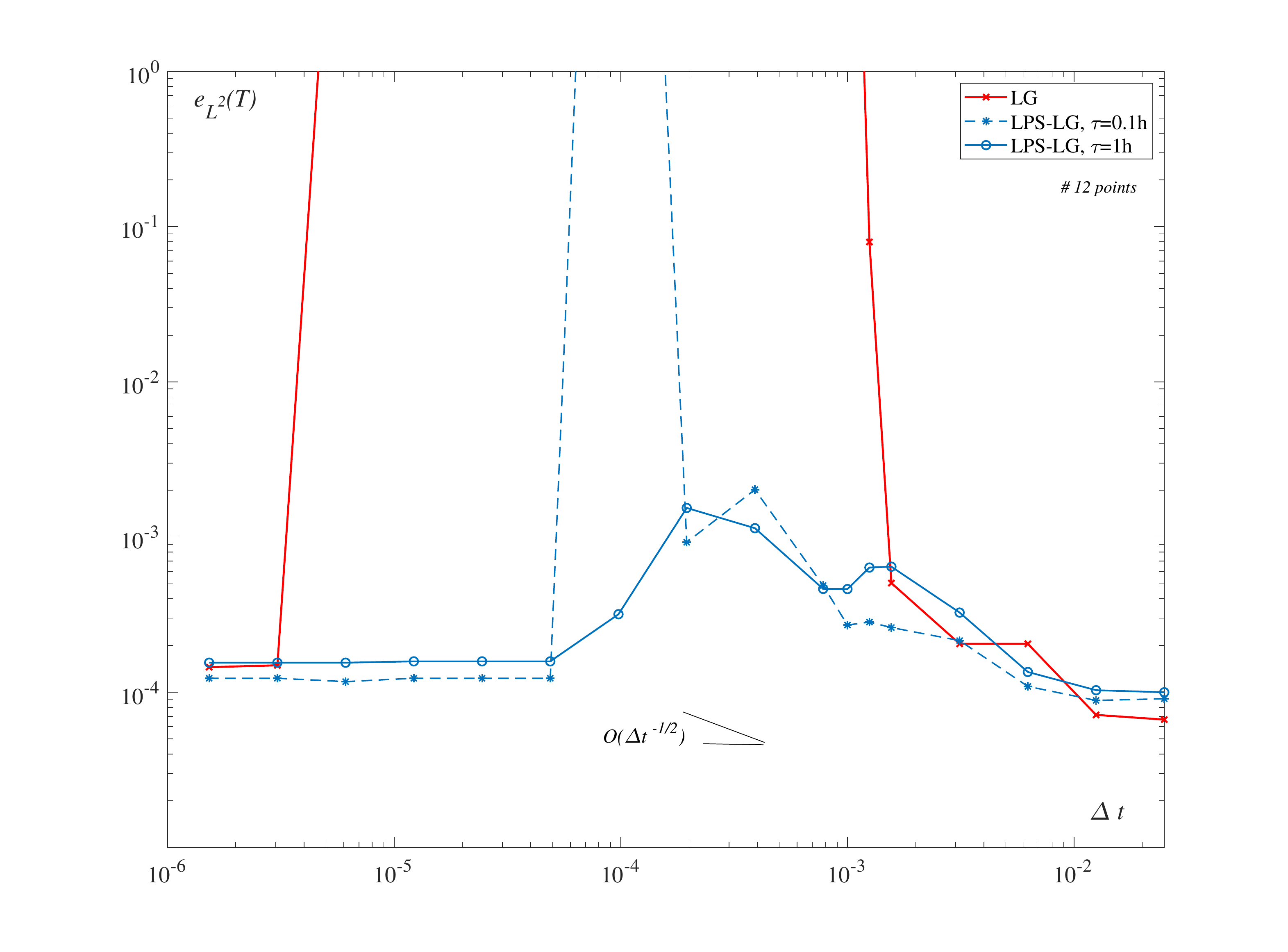}} \ \ \ %
\scalebox{0.30}{\includegraphics{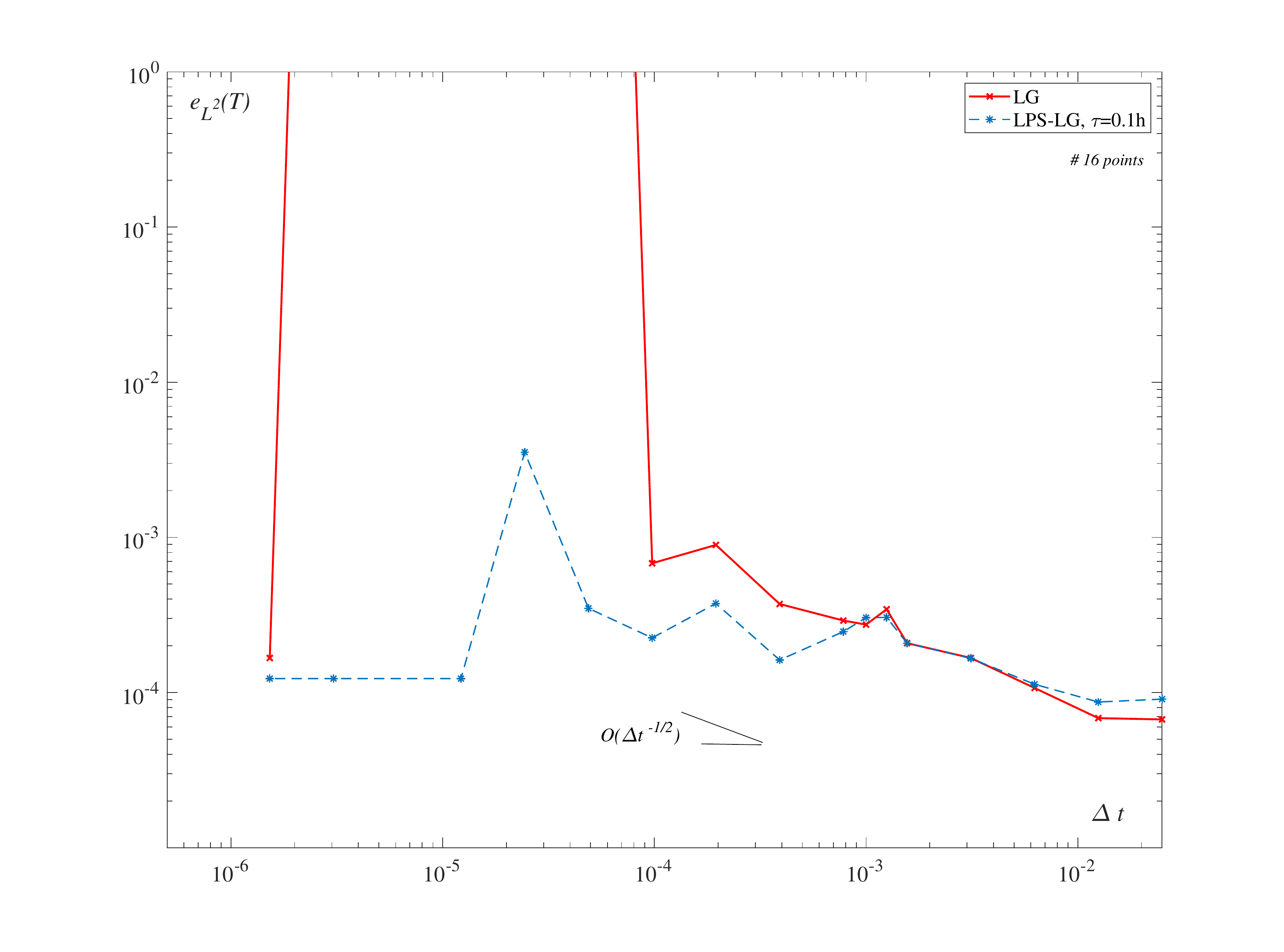}}
\end{center}
\caption{$L^2$-norm of the error in the rotating hump problem of the
two-level LPS-LG method and the conventional LG method .}
\label{figure:lps2}
\end{figure}


Figure \ref{figure:3} displays the $L^{2}$-norm of the error as a function
of $\Delta t$ for the one level LPS-LG method, the discrete spaces of which
are shown in the right panel of Figure \ref{figure:lps1}, and for the
conventional LG method with finite element space $W_{h}=P_{1}^{+}$. The mesh
size of this experiment is $h=\sqrt{2}\times 10^{-2}$. The solid lines
represent the error for the LPS-LG method with $\tau{\max} =\tau _{M}=0.1h$
and quadrature rules of $16$ and $25$ points, respectively. The dashed lines
correspond to the error of the conventional LG method.

We notice in these figures that for $\Delta t=O(h)$ and $h$ such that $\min
\left( 1,\frac{\Delta t^{1/2}\left\Vert \mathbf{u}\right\Vert _{L^{\infty
}(L^{2}(D))}}{h}\right) =1$, the solutions given by both the LPS-LG and LG
methods are very similar, regardless the quadrature rule. This fact agrees
with the results of Theorem \ref{teorema1} and Theorem \ref{teorlps},
because in this case the dominant term of the error in the conventional LG
method is $O({h^{m+1}}/{\Delta t^{1/2}})=O(h^{m+1/2})$, and in the LPS-LG
method the dominant term of the error is $O({h^{m+1}}/{\Delta t^{1/2}+}\tau
_{\max }^{1/2}\times (h^{m}+h^{s+1}))$, so letting $\tau _{\max
}^{1/2}=ch^{1/2}$ and $s=m-1$ one has that the error is also $O(h^{m+1/2})$.
However, for $\Delta t$ small enough so that $\min \left( 1,\frac{\Delta
t^{1/2}\left\Vert \mathbf{u}\right\Vert _{L^{\infty }(L^{2}(D))}}{h}\right) =%
\frac{\Delta t^{1/2}\left\Vert \mathbf{u}\right\Vert _{L^{\infty }(L^{2}(D))}%
}{h}$, the maximum error in the $L^{2}$-norm for the conventional LG method
is $O(h^{m}) $, whereas the maximum error of the LPS-LG method in the mesh
dependent norm$\ $is also $O(h^{m})$; however, since%
\begin{equation*}
\left\vert \left\vert \left\vert e^{N}\right\vert \right\vert \right\vert
^{2}=\left\Vert e\right\Vert ^{2}+\Delta t\sum_{n=1}^{N}\left( \sum_{M}\tau
_{M}(\kappa _{M}\nabla e^{n},\kappa _{M}\nabla e^{n})_{M}\right) ,
\end{equation*}%
then the $L^{2}$-error of the LPS-LG method is smaller than the $L^{2}$%
-error of the conventional LG method, and this is what we observe in Figure %
\ref{figure:3}.


\begin{figure}[th!]
\begin{center}
\scalebox{0.5}{\includegraphics{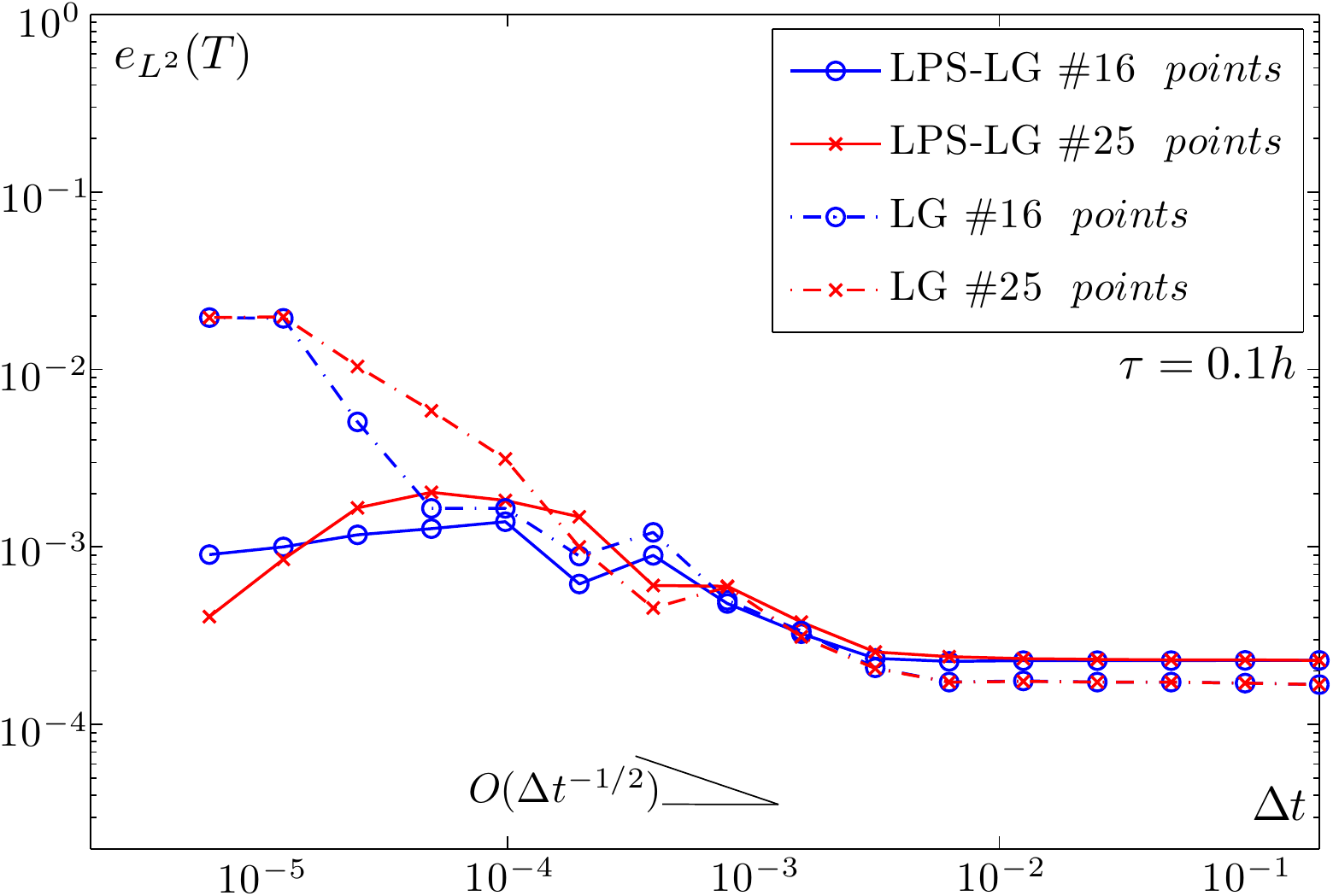}}
\end{center}
\caption{$L^2$-norm of the error in the rotating hump problem of the LPS-LG
method (solid lines) and the LG method (dashed lines).}
\label{figure:3}
\end{figure}


\section{The DC-LG method}

Numerical experiments show that when the analytical solution $c(x,t)$ is not
sufficiently smooth the LG methods presented in the previous sections are
not free from wiggles. Following the approach of \cite{john}, where the so
called shock-capturing characteristic streamline-diffusion method is
developed, but scaling the non linear dissipative term as in \cite{Naz}, we
formulate a LG method that is stable in the maximum norm with linear finite
elements, although numerical experiments show that the method may also be
stable with quadratic elements; this stabilization is achieved by adding a
non linear dissipative term on the left side of the formulation (\ref{ecua4}%
), thus obtaining the so called discontinuity-capturing LG method. In this
method, we calculate $c_{h}^{n}$ as solution of
\begin{equation}
\left( c_{h}^{n}-c_{h}^{\ast n-1},v_{h}\right) +\Delta t\sum_{K}\left(
\varepsilon _{K}(c_{h}^{n})\nabla c_{h}^{n},\nabla v_{h}\right) _{K}=0,
\label{n1}
\end{equation}%
where $\left( \varepsilon _{K}(c_{h}^{n})\nabla c_{h}^{n},\nabla
v_{h}\right) _{K}:=\int_{K}\varepsilon _{K}(c_{h}^{n})\nabla c_{h}^{n}\cdot
\nabla v_{h}dx$ and
\begin{equation}
\varepsilon _{K}(c_{h}^{n}):=C_{\varepsilon }h_{K}^{\alpha
}|R(c_{h}^{n})||_{K}\equiv C_{\varepsilon }h_{K}^{\alpha }\left. \frac{%
\left\vert c_{h}^{n}-c_{h}^{\ast n-1}\right\vert }{\Delta t}\right\vert _{K}.
\label{n2}
\end{equation}%
Here, $C_{\varepsilon }<1$ is a user-defined positive constant, the
coefficient $\alpha \in \lbrack 1,2)$ and $|R(c_{h}^{n})||_{K}$ denotes the
absolute value of the residual, restricted to the element $K$, generated by
the discretization of the material derivative along the characteristic
curves. The existence of a solution of (\ref{n1}) can be proven making use
of Corollary 1.1 of Chapter IV of \cite{GR} as in \cite{Naz}. Notice that
the amount of artificial diffusion is externally controlled by $%
C_{\varepsilon }$, $h$ and the parameter $\alpha $, the latter must be less
than 2 in order for the method to be stable in the maximum norm when the
finite element space is linear.

\subsection{Analysis of the DC-LG method}

First, we study the stability of (\ref{n1}) in both the $L^{2}$ norm and the
$L^{\infty }$ norm.


\begin{lemma}
\label{stabsc} For all $N\geq 1$, it holds
\begin{equation}
\left\Vert c_{h}^{N}\right\Vert ^{2}+\sum_{n=1}^{N}\left\Vert
c_{h}^{n}-c_{h}^{\ast n-1}\right\Vert ^{2}+2\Delta
t\sum_{n=1}^{N}\sum_{K}\left\Vert \varepsilon _{K}(c_{h}^{n})^{1/2}\nabla
c_{h}^{n}\right\Vert _{K}^{2}\leq \left\Vert c_{h}^{0}\right\Vert ^{2}.
\label{n3}
\end{equation}
\end{lemma}

\begin{proof}
Letting $v_{h}=c_{h}^{n}$ in (\ref{n1}) and taking into account that $%
\left\Vert c_{h}^{\ast n-1}\right\Vert =\left\Vert c_{h}^{n-1}\right\Vert $,
it follows that%
\begin{equation*}
\left\Vert c_{h}^{n}\right\Vert ^{2}+\left\Vert c_{h}^{n}-c_{h}^{\ast
n-1}\right\Vert ^{2}-\left\Vert c_{h}^{n-1}\right\Vert ^{2}+2\Delta
t\sum_{K}\left\Vert \varepsilon _{K}(c_{h}^{n})^{1/2}\nabla
c_{h}^{n}\right\Vert _{K}^{2}.
\end{equation*}%
Hence, it follows that%
\begin{equation*}
\left\Vert c_{h}^{N}\right\Vert ^{2}+\sum_{n=1}^{N}\left\Vert
c_{h}^{n}-c_{h}^{\ast n-1}\right\Vert ^{2}+2\Delta
t\sum_{n=1}^{N}\sum_{K}\left\Vert \varepsilon _{K}(c_{h}^{n})^{1/2}\nabla
c_{h}^{n}\right\Vert _{K}^{2}\leq \left\Vert c_{h}^{0}\right\Vert ^{2}.
\end{equation*}
\end{proof}

\begin{lemma}
\label{stabmaxn} There is a constant $C$ independent of $h$, $\Delta t$, and
$n$, but depending on the constant $C_{\varepsilon }$, such that for all $n,$

\begin{equation}
\left\Vert c_{h}^{n}\right\Vert _{L^{\infty }(D)}\leq (1+Ch^{\frac{1}{2}%
(2-\alpha )}\log (\frac{1}{h}))\left\Vert c_{h}^{0}\right\Vert _{L^{\infty
}(D)}.  \label{n4}
\end{equation}
\end{lemma}

Noting that for $p\geq 1$, $\left\Vert c_{h}^{\ast n-1}\right\Vert
_{L^{p}(D)}=\left\Vert c_{h}^{n-1}\right\Vert _{L^{p}(D)}$, we can prove
this lemma by using the the same arguments as those employed to prove Lemma
6 in \cite{Naz}. See also the proof presented in \cite{john} of the
stability in the maximum norm for the shock-capturing characteristic
streamline-diffusion method. It is worth remarking that maximum norm
stability has only been proven for linear finite elements, because this
proof makes use of a result of \cite{Szp}, which says that there is a
constant $c$ independent of $p=2a,\ a=1,2\ldots $, such that for all $%
w_{h}\in W_{h}$%
\begin{equation*}
\int_ {\mathbb{R}^{d}}\nabla w_{h}\cdot \nabla \Pi _{h}(w_{h})^{p-1}dx=\frac{%
c}{p^{2}}\sum_{K}\int_{K}\left\vert \nabla w_{h}\right\vert \left(
w_{h}\right) ^{p-2}dx.
\end{equation*}%
And this result has only been proven for linear finite elements. However,
via numerical examples, we have observed that the maximum norm stability
also holds in cases where the solution exhibits strong discontinuities for
quadratic elements.

For the error analysis we have the following result.

\begin{theorem}
\label{teorema3}Let $c\in L^{\infty }(H^{m+1}(D)\cap H_{0}^{1}(D))\cap
L^{\infty }(W^{1,\infty }(D))$. Then, there exists a constant $C$
independent of $\Delta t,\ h$ and $n$, but depending on $\left\vert
D\right\vert $, $\left\vert c\right\vert _{L^{\infty }(H^{m+1}(D))}$ and $%
\left\Vert \nabla c\right\Vert _{L^{\infty }(L^{\infty }(D))}$, such that
\begin{equation}
\left\Vert c-c_{h}\right\Vert _{l^{\infty }\left( L^{2}(D)\right) }\leq
\frac{C(h^{m+1}+C_{\varepsilon }h^{\alpha })}{\Delta t^{1/2}}.  \label{teor3}
\end{equation}
\end{theorem}

\begin{proof}
The error equation is%
\begin{equation}
\left( e^{n}-e^{\ast n-1},v_{h}\right) -\Delta t\sum_{K}(\varepsilon
_{K}(c_{h}^{n})\nabla c_{h}^{n},\nabla v_{h})_{K}=0  \label{dce1}
\end{equation}%
Noting that $c_{h}^{n}=c^{n}-e^{n}$ we have that%
\begin{equation*}
\begin{array}{r}
\Delta t\sum_{K}(\varepsilon _{K}(c_{h}^{n})\nabla c_{h}^{n},\nabla
v_{h})_{K}=\Delta t\displaystyle\sum_{K}(\varepsilon _{K}(c_{h}^{n})\nabla
c^{n},\nabla v_{h})_{K} \\
\\
-\Delta t\displaystyle\sum_{K}(\varepsilon _{K}(c_{h}^{n})\nabla
e^{n},\nabla v_{h})_{K},%
\end{array}%
\end{equation*}%
so we can write the error equation as%
\begin{equation}
\left( e^{n}-e^{\ast n-1},v_{h}\right) +\Delta t\sum_{K}(\varepsilon
_{K}(c_{h}^{n})\nabla e^{n},\nabla v_{h})_{K}=\Delta t\sum_{K}(\varepsilon
_{K}(c_{h}^{n})\nabla c^{n},\nabla v_{h})_{K}.  \label{dce2}
\end{equation}%
Now, setting in this equation $v_{h}=\theta _{h}^{n}=e^{n}-\rho ^{n}$, where
$\rho ^{n}=c^{n}-$ $\Pi _{h}c^{n}$ and $\theta _{h}^{n}=\Pi
_{h}c^{n}-c_{h}^{n}$, we get%
\begin{equation}
\begin{array}{r}
\left( e^{n}-e^{\ast n-1},e^{n}\right) +\Delta t\sum_{K}(\varepsilon
_{K}(c_{h}^{n})\nabla e^{n},\nabla e^{n})_{K}=\left( e^{n}-e^{\ast n-1},\rho
^{n}\right) \\
\\
+\Delta t\sum_{K}(\varepsilon _{K}(c_{h}^{n})\nabla e^{n},\nabla \rho
^{n})_{K} \\
\\
+\Delta t\sum_{K}(\varepsilon _{K}(c_{h}^{n})\nabla e^{n},\nabla c^{n})_{K}
\\
\\
-\Delta t\sum_{K}(\varepsilon _{K}(c_{h}^{n})\nabla c^{n},\nabla \rho
^{n})_{K}\equiv \sum_{i=1}^{4}R_{i}.%
\end{array}
\label{dce3}
\end{equation}%
We estimate the terms $R_{i}$ on the right hand side. Thus, regarding $R_{1}$
we apply the Cauchy-Schwarz inequality to obtain that%
\begin{equation*}
\left\vert R_{1}\right\vert =\left\vert \left( e^{n}-e^{\ast n-1},\rho
^{n}\right) \right\vert \leq \frac{1}{8}\left\Vert e^{n}-e^{\ast
n-1}\right\Vert ^{2}+2\left\Vert \rho ^{n}\right\Vert ^{2}.
\end{equation*}%
As for the term $R_{2}$, we use the same inequality to get%
\begin{equation*}
\left\vert R_{2}\right\vert \leq (\delta /2)\Delta t\sum_{K}(\varepsilon
_{K}(c_{h}^{n})\nabla e^{n},\nabla e^{n})_{K}+(2\delta )^{-1}\Delta
t\sum_{K}(\varepsilon _{K}(c_{h}^{n})\nabla \rho ^{n},\nabla \rho
^{n})_{K},\ 0<\delta <1.
\end{equation*}%
Similarly%
\begin{equation*}
\left\vert R_{3}\right\vert \leq (\delta /2)\Delta t\sum_{K}(\varepsilon
_{K}(c_{h}^{n})\nabla e^{n},\nabla e^{n})_{K}+(2\delta )^{-1}\Delta
t\sum_{K}(\varepsilon _{K}(c_{h}^{n})\nabla c^{n},\nabla c^{n})_{K},
\end{equation*}%
and%
\begin{equation*}
\left\vert R_{4}\right\vert \leq \frac{1}{2}\Delta t\sum_{K}(\varepsilon
_{K}(c_{h}^{n})\nabla c^{n},\nabla c^{n})_{K}+\frac{1}{2}\Delta
t\sum_{K}(\varepsilon _{K}(c_{h}^{n})\nabla \rho ^{n},\nabla \rho ^{n})_{K}.
\end{equation*}%
Substituting this estimates in (\ref{dce3}) with $\delta =1/2$ yields%
\begin{equation}
\begin{array}{r}
\left\Vert e^{n}\right\Vert ^{2}+\frac{3}{4}\left\Vert e^{n}-e^{\ast
n-1}\right\Vert ^{2}-\left\Vert e^{n-1}\right\Vert ^{2}+\Delta
t\sum_{K}(\varepsilon _{K}(c_{h}^{n})\nabla e^{n},\nabla e^{n})_{K} \\
\\
\leq 4\left\Vert \rho ^{n}\right\Vert ^{2}+3\Delta t\sum_{K}(\varepsilon
_{K}(c_{h}^{n})\nabla c^{n},\nabla c^{n})_{K} \\
\\
+3\Delta t\sum_{K}(\varepsilon _{K}(c_{h}^{n})\nabla \rho ^{n},\nabla \rho
^{n})_{K}.%
\end{array}
\label{dce4}
\end{equation}%
Next, we have to estimate the last two terms on the right hand side of this
inequality.%
\begin{equation*}
\begin{array}{r}
\Delta t\sum_{K}(\varepsilon _{K}(c_{h}^{n})\nabla c^{n},\nabla
c^{n})_{K}=C_{\varepsilon }\Delta t\displaystyle\sum_{K}h_{K}^{\alpha
}\int_{K}\frac{\left\vert c_{h}^{n}-c_{h}^{\ast n-1}\right\vert }{\Delta t}%
\left( \nabla c^{n}\right) ^{2}dK \\
\\
\leq C_{\varepsilon }\left\Vert \nabla c^{n}\right\Vert _{L^{\infty }(D)}^{2}%
\displaystyle\sum_{K}h_{K}^{\alpha }\int_{K}\left\vert c_{h}^{n}-c_{h}^{\ast
n-1}\right\vert dK \\
\\
\leq C_{\varepsilon }\left\Vert \nabla c^{n}\right\Vert _{L^{\infty }(D)}^{2}%
\displaystyle\sum_{K}h_{K}^{\alpha }\left\Vert c_{h}^{n}-c_{h}^{\ast
n-1}\right\Vert _{L^{1}(K)} \\
\\
\leq C_{\varepsilon }\left\Vert \nabla c^{n}\right\Vert _{L^{\infty
}(D)}^{2}h^{\alpha }\left\Vert c_{h}^{n}-c_{h}^{\ast n-1}\right\Vert
_{L^{1}(D)}.%
\end{array}%
\end{equation*}%
It remains to estimate the term $\left\Vert c_{h}^{n}-c_{h}^{\ast
n-1}\right\Vert _{L^{1}(D)}$. To do so, we observe that%
\begin{equation*}
c_{h}^{n}-c_{h}^{\ast n-1}=\left( c^{n}-e^{n}\right) -\left( c^{\ast
n-1}-e^{\ast n-1}\right) =e^{n}-e^{\ast n-1}
\end{equation*}%
because $c^{n}=c^{\ast n-1}$, then $\left\Vert c_{h}^{n}-c_{h}^{\ast
n-1}\right\Vert _{L^{1}(D)}=\left\Vert e^{n}-e^{\ast n-1}\right\Vert
_{L^{1}(D)}$ and by virtue of the Cauchy-Schwarz inequality $\left\Vert
e^{n}-e^{\ast n-1}\right\Vert _{L^{1}(D)}\leq C_{D}\left\Vert e^{n}-e^{\ast
n-1}\right\Vert _{L^{2}(D)}$, $C_{D}=C(|D|)$, $|D|$ being the measure of $D$%
; hence
\begin{equation*}
C_{\varepsilon }\left\Vert \nabla c^{n}\right\Vert _{L^{\infty
}(D)}^{1}h^{\alpha }\left\Vert c_{h}^{n}-c_{h}^{\ast n-1}\right\Vert
_{L^{1}(D)}\leq \frac{1}{16}\left\Vert e^{n}-e^{\ast n-1}\right\Vert
_{L^{2}(D)}^{2}+4C_{\varepsilon }^{2}h^{2\alpha }C_{D}^{2}\left\Vert \nabla
c^{n}\right\Vert _{L^{\infty }(D)}^{4}.
\end{equation*}%
Therefore, we can set that
\begin{equation}
\Delta t\sum_{K}(\varepsilon _{K}(c_{h}^{n})\nabla c^{n},\nabla
c^{n})_{K}\leq \frac{1}{16}\left\Vert e^{n}-e^{\ast n-1}\right\Vert
_{L^{2}(D)}^{2}+4C_{\varepsilon }^{2}h^{2\alpha }C_{D}^{2}\left\Vert \nabla
c^{n}\right\Vert _{L^{\infty }(D)}^{4}.  \label{dce5}
\end{equation}%
We estimate now the term $\Delta t\sum_{K}(\varepsilon _{K}(c_{h}^{n})\nabla
\rho ^{n},\nabla \rho ^{n})_{K}$. To this end, we notice that%
\begin{equation*}
\Delta t\sum_{K}(\varepsilon _{K}(c_{h}^{n})\nabla \rho ^{n},\nabla \rho
^{n})_{K}=C_{\varepsilon }\sum_{K}h_{K}^{\alpha }\int_{K}\left\vert
c_{h}^{n}-c_{h}^{\ast n-1}\right\vert \left( \nabla \rho ^{n}\right) ^{2}dK,
\end{equation*}%
but by virtue of (\ref{eq4e})\ we have that $\left\Vert \nabla \rho
^{n}\right\Vert _{L^{\infty }(D)}\leq c_{3}\left\Vert \nabla
c^{n}\right\Vert _{L^{\infty }(D)}$, then we can write that
\begin{equation*}
\Delta t\sum_{K}(\varepsilon _{K}(c_{h}^{n})\nabla \rho ^{n},\nabla \rho
^{n})_{K}\leq c_{3}^{2}C_{\varepsilon }\left\Vert \nabla c^{n}\right\Vert
_{L^{\infty }(D)}^{2}\sum_{K}h_{K}^{\alpha }\int_{K}\left\vert
c_{h}^{n}-c_{h}^{\ast n-1}\right\vert dK;
\end{equation*}

so, arguing as we have just done for the term $\Delta t\sum_{K}(\varepsilon
_{K}(c_{h}^{n})\nabla c^{n},\nabla c^{n})_{K}$ it follows that%
\begin{equation}
\Delta t\sum_{K}(\varepsilon _{K}(c_{h}^{n})\nabla \rho ^{n},\nabla \rho
^{n})_{K}\leq \frac{1}{16}\left\Vert e^{n}-e^{\ast n-1}\right\Vert
_{L^{2}(D)}^{2}+4c_{3}^{4}C_{\varepsilon }^{2}h^{2\alpha}C_{D}^{2}\left\Vert
\nabla c^{n}\right\Vert _{L^{\infty }(D)}^{4}.  \label{dce6}
\end{equation}

Substituting (\ref{dce5}) and (\ref{dce6}) in (\ref{dce4}) and using the
estimate (\ref{eq4e}) yields%
\begin{equation*}
\begin{array}{r}
\left\Vert e^{n}\right\Vert ^{2}+\frac{1}{2}\left\Vert e^{n}-e^{\ast
n-1}\right\Vert ^{2}-\left\Vert e^{n-1}\right\Vert ^{2}+\Delta
t\sum_{K}(\varepsilon _{K}(c_{h}^{n})\nabla e^{n},\nabla e^{n})_{K} \\
\\
\leq C\left( h^{2(m+1)}+C_{\varepsilon }^{2}h^{2\alpha }\right) .%
\end{array}%
\end{equation*}%
where $C$ is a constant that depends on $C_{D}$, $\left\vert
c^{n}\right\vert _{H^{m+1}(D)}$ and $\left\Vert \nabla c^{n}\right\Vert
_{L^{\infty }(D)}$. Summing both terms of this inequality from $n=1$ up to $%
n=N$ it follows that%
\begin{equation*}
\begin{array}{r}
\left\Vert e^{N}\right\Vert ^{2}+\frac{1}{2}\sum_{n=1}^{N}\left\Vert
e^{n}-e^{\ast n-1}\right\Vert ^{2}+\Delta
t\sum_{n=1}^{N}\sum_{K}(\varepsilon _{K}(c_{h}^{n})\nabla e^{n},\nabla
e^{n})_{K} \\
\\
\leq \frac{C}{\Delta t}\left( h^{2(m+1)}+C_{\varepsilon }^{2}h^{2\alpha
}\right) ,%
\end{array}%
\end{equation*}%
or equivalently%
\begin{equation*}
\left\Vert e^{N}\right\Vert \leq C\frac{\left( h^{m+1}+C_{\varepsilon
}h^{\alpha }\right) }{\Delta t^{1/2}}.
\end{equation*}
\end{proof}

\begin{remark}
This estimate depends on $\Delta t^{-1/2}$ so that for $h$ fixed blows up as
$\Delta t\rightarrow 0$. In contrast with the previous LG methods, for the
DC-LG method we have not been able to find an error estimate free from the $%
\Delta t^{-1/2}$ dependence; however, based on numerical experiments and
assuming that the maximum norm stability holds, we may hypothesizes that
there is a $\Delta t_{c}$, such that for $\Delta t\leq \Delta t_{c}$ the
error will not increase, remaining nearly constant or decreasing very
slowly. So, noting that $c_{h}^{n}-c_{h}^{\ast n-1}=e^{n}-e^{\ast n-1}$
because $c^{n}=c^{\ast n-1}$, we can argue that for $\Delta t\leq \Delta
t_{c}$%
\begin{equation*}
\max_{n}\max_{K}\left. \frac{\left\vert e^{n}-e^{\ast n-1}\right\vert }{%
\Delta t}\right\vert _{K}=\beta ,
\end{equation*}%
$\beta $ being a small constant that depends on $m$; hence, we can consider
that $\varepsilon _{K}(c_{h}^{n})$ is a constant, specifically, for all $K$
and $n$ we set
\begin{equation*}
\nu :=\varepsilon _{K}(c_{h}^{n})=C_{\varepsilon }\beta h^{\alpha }.
\end{equation*}%
Then, the error equation can be written now as%
\begin{equation}
\left( e^{n}-e^{\ast n-1},v_{h}\right) -\Delta t\nu \sum_{K}(\nabla
c_{h}^{n},\nabla v_{h})_{K}=0.  \label{rem1}
\end{equation}%
So, as we have done above, we let%
\begin{equation*}
v_{h}=\theta _{h}^{n},\ e^{n}=\rho ^{n}+\theta _{h}^{n},\ e^{\ast n-1}=\rho
^{\ast n-1}+\theta _{h}^{\ast n-1}\ \mathrm{and\ }c_{h}^{n}=c^{n}-(\rho
^{n}+\theta _{h}^{n}),
\end{equation*}%
with $\rho ^{n}=c^{n}-P_{h}c^{n}$, and recast (\ref{rem1}) as%
\begin{equation}
\begin{array}{r}
\left( \theta _{h}^{n}-\theta _{h}^{\ast n-1},\theta _{h}^{n}\right) +\Delta
t\nu \nabla \theta _{h}^{n},\nabla \theta _{h}^{n})=-(\rho ^{n}-\rho ^{\ast
n-1},\theta _{h}^{n}) \\
\\
-\Delta t\nu (\nabla \rho ^{n},\nabla \theta _{h}^{n})+\Delta t\nu (\nabla
c^{n},\nabla \theta _{h}^{n}).%
\end{array}
\label{rem2}
\end{equation}%
If we compares this equation with (\ref{11}), we can consider that the
artificial dissipation terms represent a perturbation to the equation of the
pure advection problem, so, we can expect that when $\nu \rightarrow 0$ (\ref%
{rem2}) will yield the same estimate as (\ref{11}). To check that this is
the case, we bound the terms $-(\rho ^{n}-\rho ^{\ast n-1},\theta _{h}^{n})$%
, $(\nabla \rho ^{n},\nabla \theta _{h}^{n})$ and $(\nabla c^{n},\nabla
\theta _{h}^{n})$ as we have done many times before and can easily arrive to
the estimate%
\begin{equation*}
\left\Vert \theta _{h}^{N}\right\Vert \leq C\left( h^{m}+\nu ^{1/2}\right) ,
\end{equation*}%
where the constant \ $C$ depends on $\left\vert c\right\vert _{L^{\infty
}(0,T;H^{m+1}(D))}$, and consequently,%
\begin{equation*}
\left\Vert e^{N}\right\Vert =O\left( h^{m}+\left( C_{\varepsilon }\beta
h^{\alpha }\right) ^{1/2}\right) .
\end{equation*}%
So, if $C_{\varepsilon }\beta $ is so small that $\left( C_{\varepsilon
}\beta h^{\alpha }\right) ^{1/2}\leq h^{m}$, then $\left\Vert
e^{N}\right\Vert =O(h^{m})$.
\end{remark}

\subsection{Numerical tests with the DC-LG method}

Since the method is designed to deal with discontinuous initial conditions,
we shall perform two numerical tests. The first one is again the hump
problem to see wether the error behaves according to Theorem \ref{teorema3};
the second test uses as initial condition the so called ``slotted''
cylinder, this a typical initial condition to study the ability of schemes
to deal with strong discontinuities.

\subsubsection{The hump test}

We run the test under the same conditions as the numerical test for the
conventional LG method. We show in Figures \ref{figure:4} and \ref{figure:5}
the results for the meshes with mesh parameter $h=0.05$ and $h=0.025$ after
one revolution, and with the constants $C_{\varepsilon}$ and $\alpha$ of the
expression for the artificial diffusivity (\ref{n2}) taking the values $%
C_{\varepsilon}=0.01$, $C_{\varepsilon}=0.1$ and $\alpha=\frac{3}{2}$. These
results must be compared with those of Figure \ref{figure:2}.

\begin{figure}[ht!]
\begin{center}
\scalebox{0.5}{\includegraphics{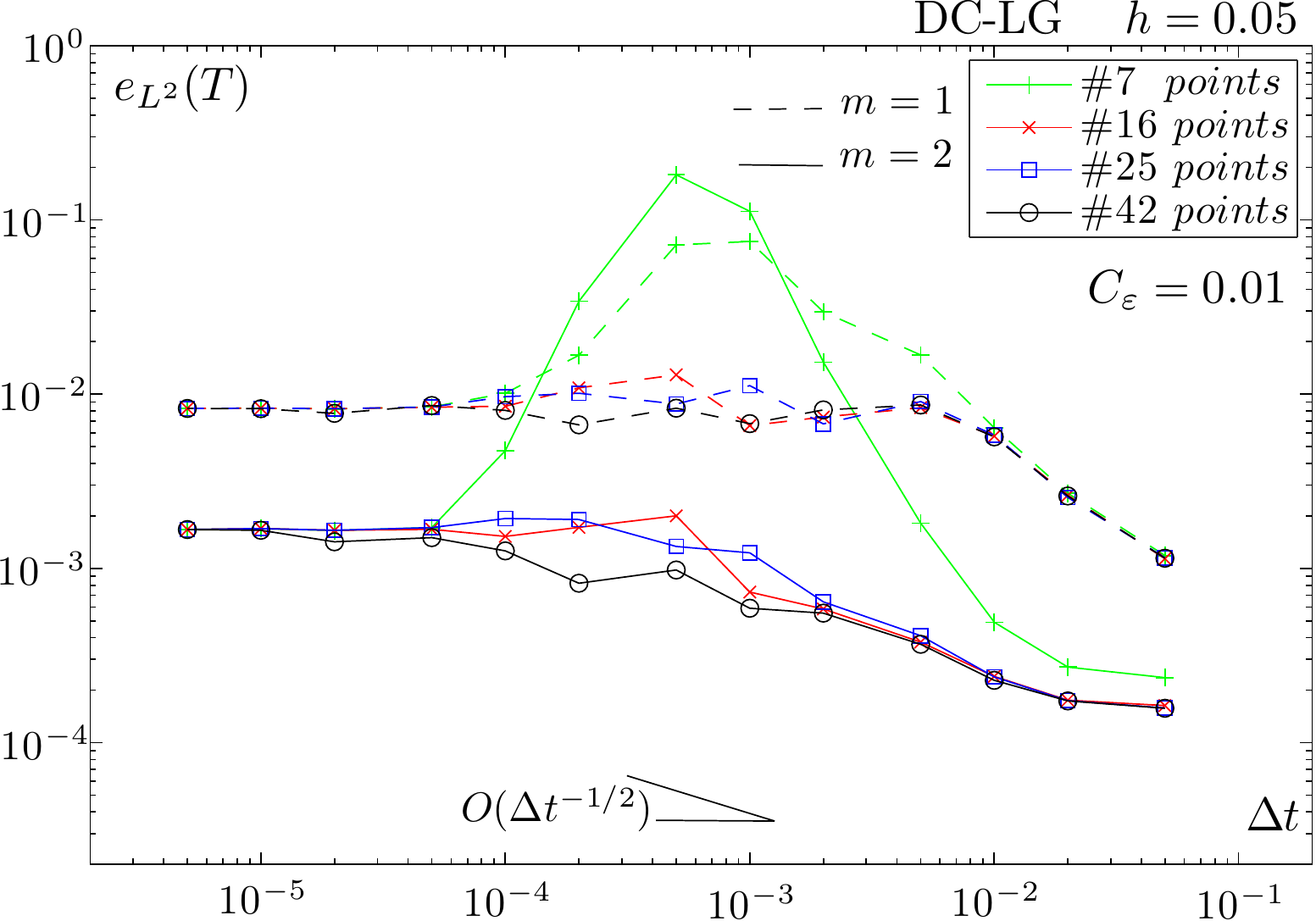}}\ \ \ %
\scalebox{0.5}{\includegraphics{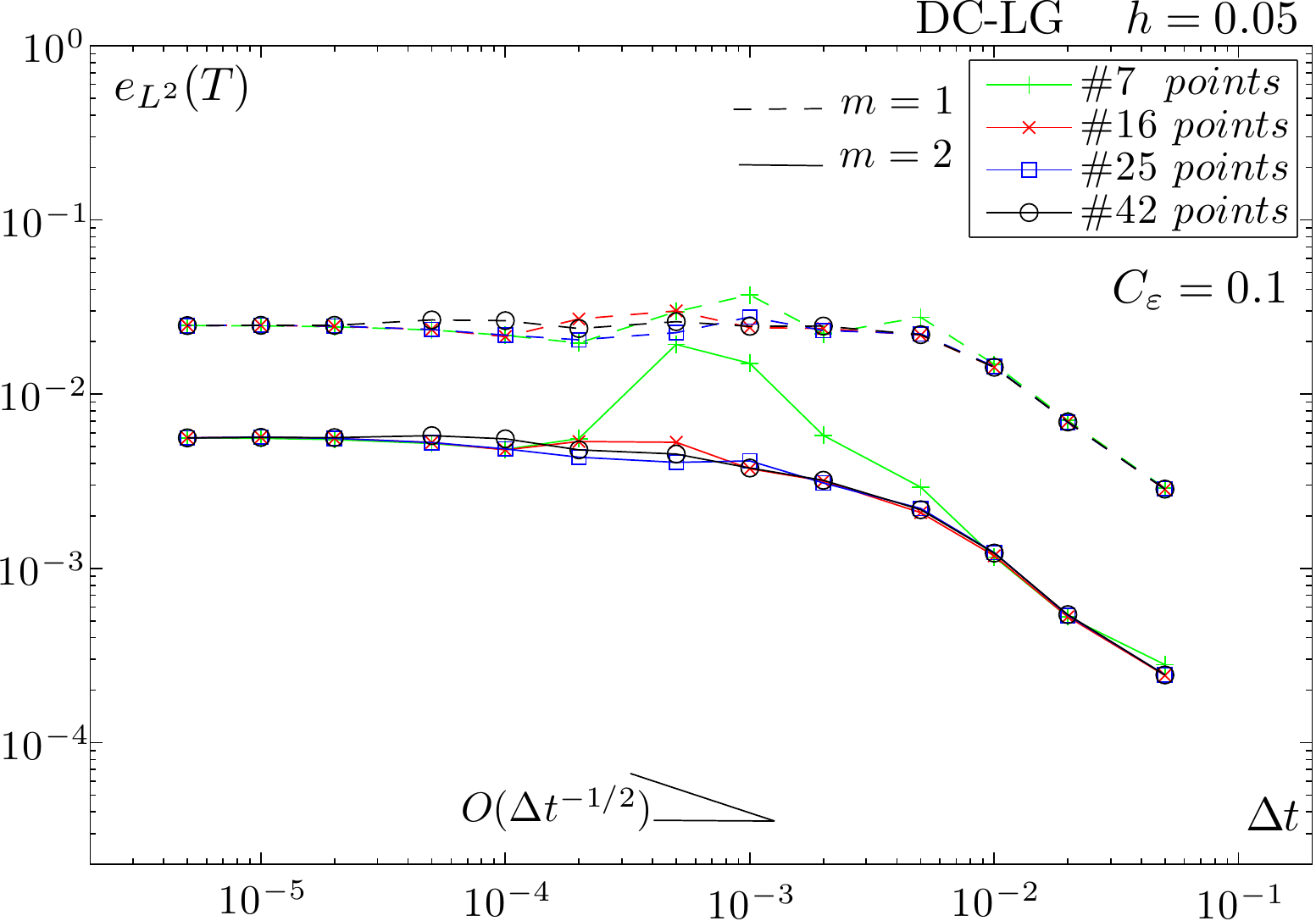}}
\end{center}
\caption{$L^2$-error norm with the DC-LG method in the rotating hump problem
for $h=0.05$}
\label{figure:4}
\end{figure}

We notice the following facts: (a) For high order quadrature rules, the
error of the DC-LG method shows a similar, but smoother, behavior as the
error of the conventional LG method, with the feature that the higher the
constant $C_{\varepsilon }$ or the coarser the mesh the smoother the profile
of the error curves; this is a consequence of the nonlinear artificial
diffusivity that depends on both $C_{\varepsilon }$ and $h$. (b) For the low
order quadrature rule of 7 points, the DC-LG method loses accuracy for those
values $\Delta t$ for which the conventional LG method is inaccurate or even
unstable; in fact, for $C_{\varepsilon }=0.01$ and $h=0.025$ there is an
interval of values $\Delta t$, which, roughly speaking, corresponds with
those values for which the conventional LG method with quadratic polynomials
becomes unstable, in which the DC-LG method with quadratic polynomials is
less accurate than with linear polynomials. This can be explained because
when both $C_{\varepsilon }$ and $h$ are low the artificial diffusivity is
not sufficiently strong to prevent the instability. (c) Roughly speaking, we
can say that the higher the artificial viscosity the less sensitive the
DC-LG method is to the order of the quadrature rules, provided that such
rules are exact for polynomials of degree $2m$. (d) For $\Delta t=O(h)$ or $%
\Delta t=o(h^{2})$, all the quadrature rules give about the same solution.
This means that in those ranges of values $\Delta t$ it is not necessary the
use of high order quadrature rules, just a rule which is exact for
polynomials of degree $2(m+1)$ would suffice. (e) Looking at the profiles of
the error curves, we notice that for high order quadrature rules the error
behaves as Theorem \ref{teorema3} says, that is, there is a value $\Delta
t_{c}$ (in this test, $\Delta t_{c}=O(h^{-2})$), such that for $\Delta t\geq
\Delta t_{c}$ the error is $O(h^{m+1}+C_{\varepsilon }h^{\alpha })/\Delta
t^{1/2}$. However, for $\Delta t<\Delta t_{c}$, the error does not grow and
remains more or less constant, particularly as the artificial diffusivity is
high enough, see Figure \ref{figure:4}. Finally, fixing the mesh and the
parameter $\alpha $, this test shows that as the constant $C_{\varepsilon }$
becomes smaller and smaller, the DC-LG solution approaches the solution of
conventional LG method.

\begin{figure}[th!]
\begin{center}
\scalebox{0.5}{\includegraphics{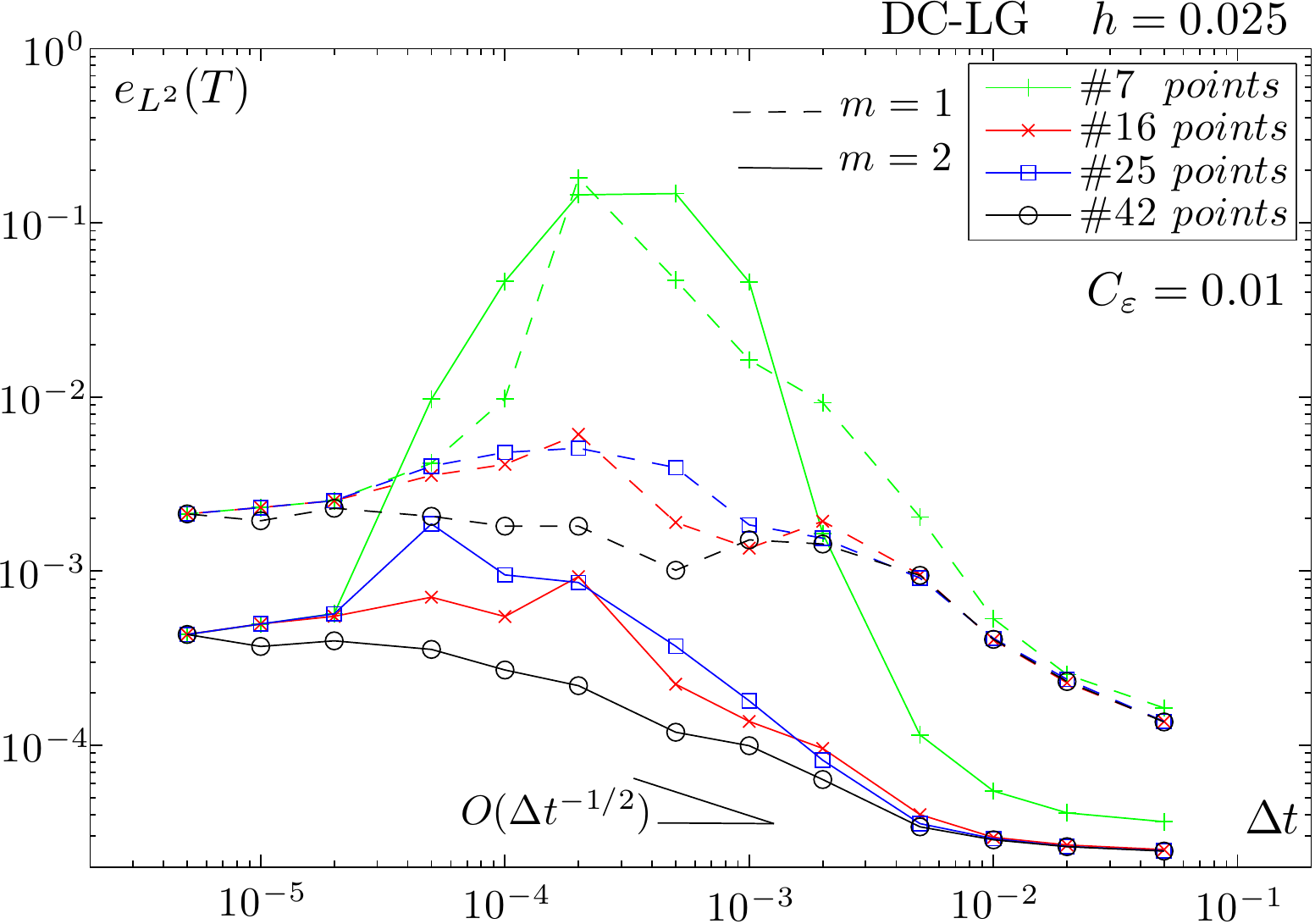}}\ \ \ %
\scalebox{0.5}{\includegraphics{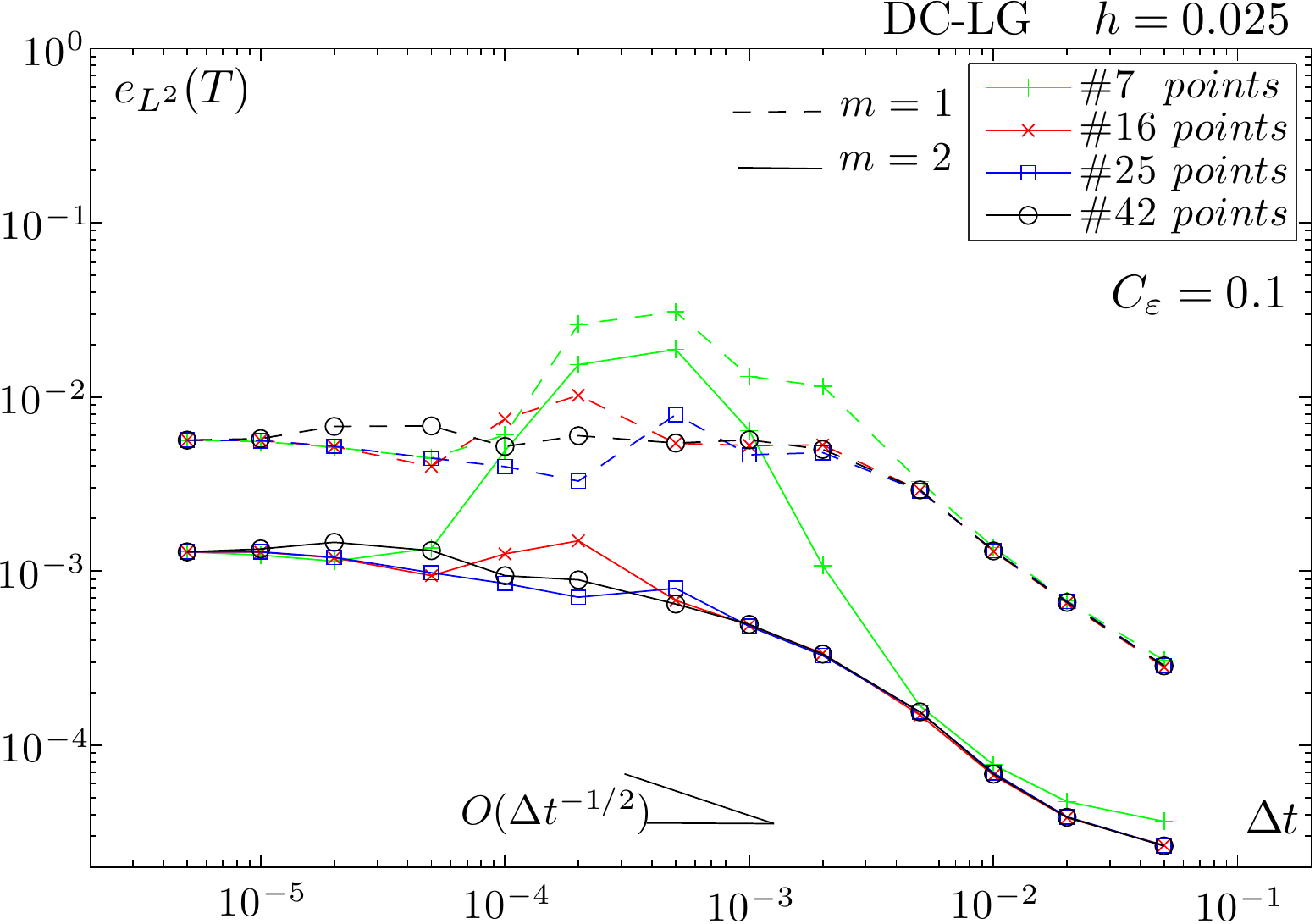}}
\end{center}
\caption{$L^2$-error norm with the DC-LG method in the rotating hump problem
for $h=0.025$}
\label{figure:5}
\end{figure}


\subsubsection{The slotted cylinder}

Our second test is the so called slotted cylinder. The idea behind this test
is to assess the ability of the DC-LG method to deal with strong
discontinuities; specifically, we wish to see how the scheme smears out an
initial condition that is strongly discontinuous. The domain $%
D:=[-1,1]\times \lbrack -1,1]$, the velocity field $\mathbf{u}$ is the same
as in the previous tests, i.e., $\mathbf{u}=2\pi (-x_{2},x_{1})$, and the
initial condition is a cylinder of height 1 and radius 0.25 centered at
(0.5,0), with a slot along the plane $x=0.5$ of width 0.1 and depth 0.35.
The simulations are carried out with a time step $\Delta t=0.01$ in the mesh
with mesh parameter $h=0.025$, and the numerical initial condition being
computed by the $L^{2}$-projection onto the finite element space $V_{h}$.
Although we are aware that this is not a good way to calculate the numerical
initial condition because, as we see in Figure \ref{figure:7}, some
overshoots and undershoots are generated by the $L^{2}$-projection, we have
left it to test the capability of DC-LG method to suppress the wiggles; it
is clear that the method is able to kill them out after few time steps when
the constant $C_{\varepsilon }$ of the artificial diffusion $\varepsilon
_{h}(c_{h}^{n})$ is $C_{\varepsilon }=0.1$. A better approach to calculate
the numerical initial condition would have been to perform $L^{2}$%
-projection of the exact initial condition with linear elements and lumped
mass matrix, yielding this way a somewhat smoother initial condition. The
integration time $T=1$. For the results, we have used the exact trajectories
and the integrals (\ref{gal_proj}) have been calculated with the quadrature
rule of 16 points. Based on the results of the hump test, we know that for
the values $\Delta t=0.01$ and $h=0.025$ the solution is not sensitive to
the order of the quadrature rules used to approximate the integrals (\ref%
{gal_proj}), provided that the rule is exact for polynomials of degree $%
>2(m+1)$.

\begin{figure}[th!]
\begin{center}
\scalebox{0.6}{\includegraphics{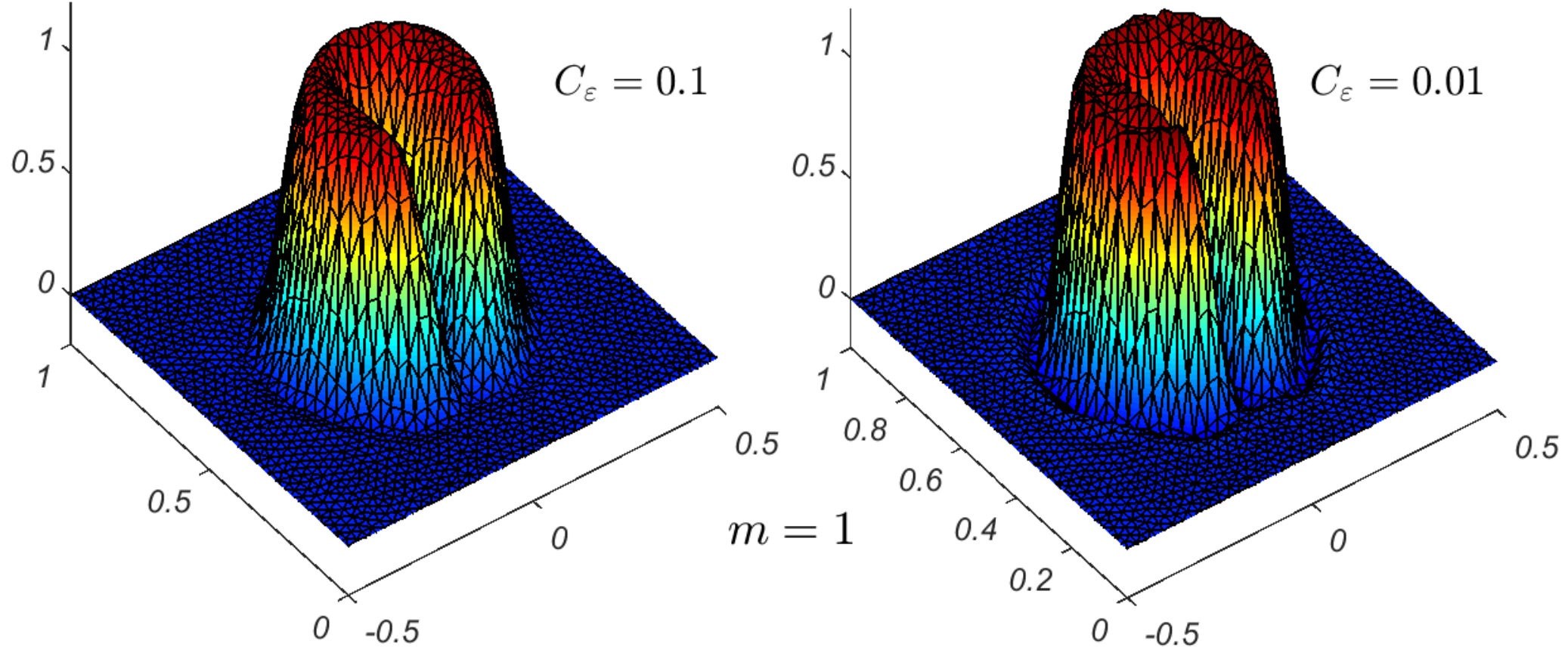}} \scalebox{0.5}{%
\includegraphics{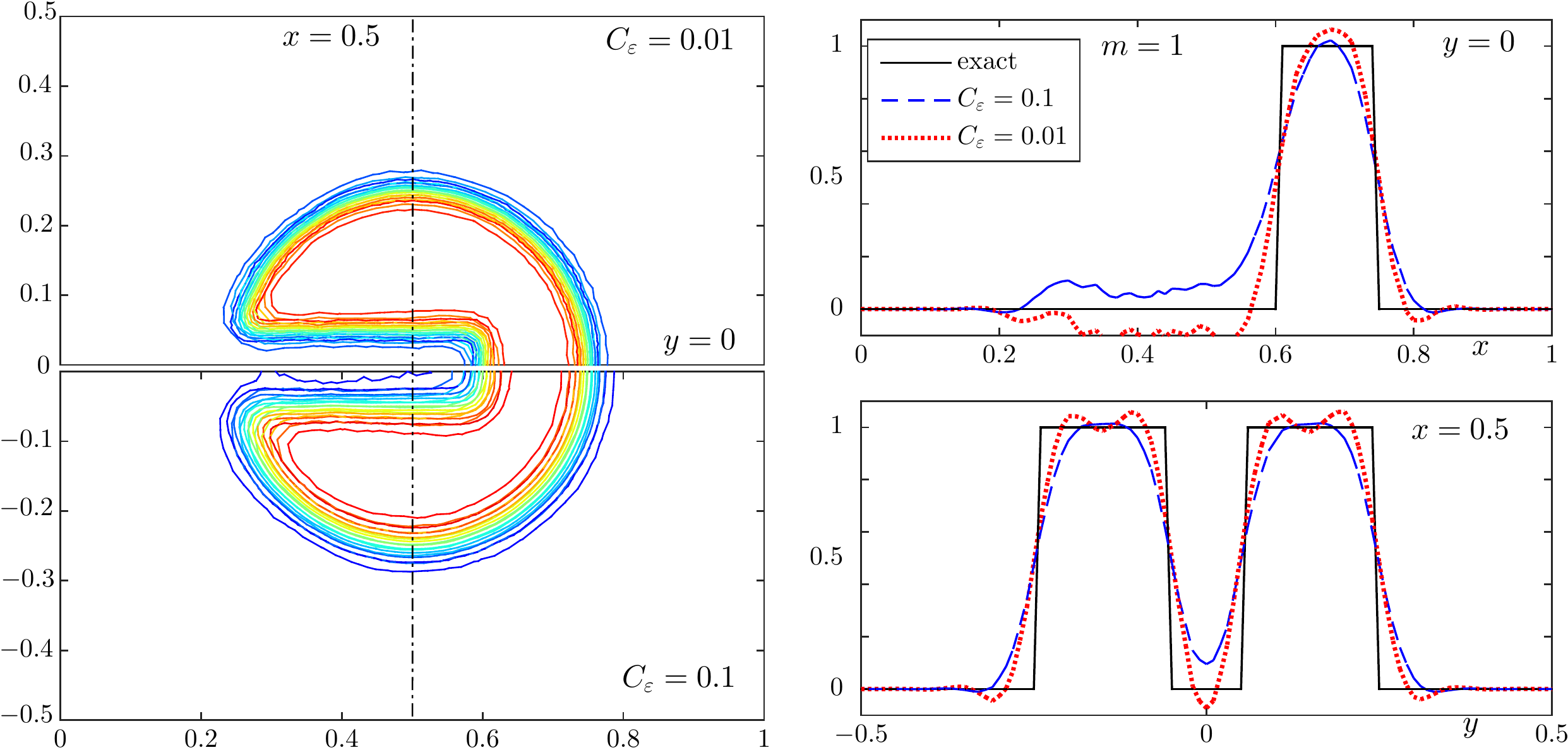}}
\end{center}
\caption{Slotted cylinder after one revolution for linear finite elements $%
m=1$. Upper panel: three dimensional view of the solutions. Lower panel: the
level lines (on the left) and cross sections (on the right) that correspond
with the figures of the upper panel.}
\label{figure:6}
\end{figure}


We display in the upper panel of Figure \ref{figure:6} a three dimensional
view of the cylinder after one revolution, whereas in the low panel are
represented the level lines (on the left) and cross sections (on the right)
when the constant of the artificial diffusion takes the values $%
C_{\varepsilon }=0.1$ and $C_{\varepsilon }=0.01$. This solution has been
calculated with linear polynomials ($m=1$). We notice that the width of the
upper face of the lobes and the width of the \textquotedblleft
bridge\textquotedblright\ as well as the depth of the slot are reasonably
well preserved for both constants $C_{\varepsilon }=0.1$ and $C_{\varepsilon
}=0.01$. It is worth remarking that the figures of the upper and middle
panel with $C_{\varepsilon }=0.1$ compare very well with those obtained in
\cite{john} and \cite{Hans} applying the shock-capturing
streamline-diffusion method with a time step $\Delta t=0.01$ and the mesh
size $h=0.01$, which is $2.5$ times smaller than the one we use. It is clear
that with the constant $C_{\varepsilon }=0.1$ the DC-LG method introduces a
major degree of smearing, and when $C_{\varepsilon }=0.01$ the method is not
able to suppress the wiggles generated around the discontinuities at the
first time step.

Similar representations of the numerical solution calculated with quadratic
polynomials (m=2) are displayed in Figure \ref{figure:7}. If we compare
these graphs with those of Figure \ref{figure:6} one sees that it is clear
the improvement of the numerical solution calculated with quadratic
polynomials; for instance, the slopes of the cylinder sides, the width of
the lobes of the upper face and the width of the ``bridge'' are much better
represented with quadratic elements than with linear elements.

\begin{figure}[th!]
\begin{center}
\scalebox{0.6}{\includegraphics{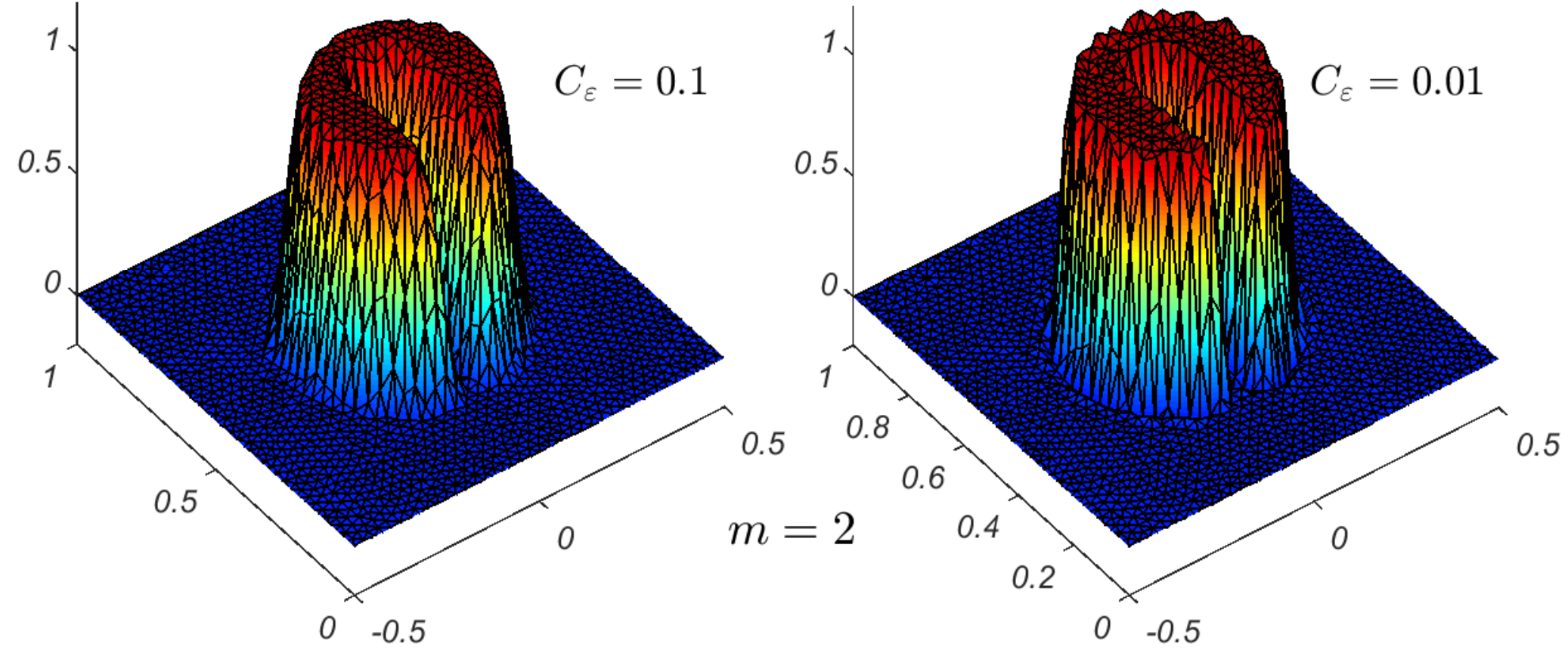}} \scalebox{0.5}{%
\includegraphics{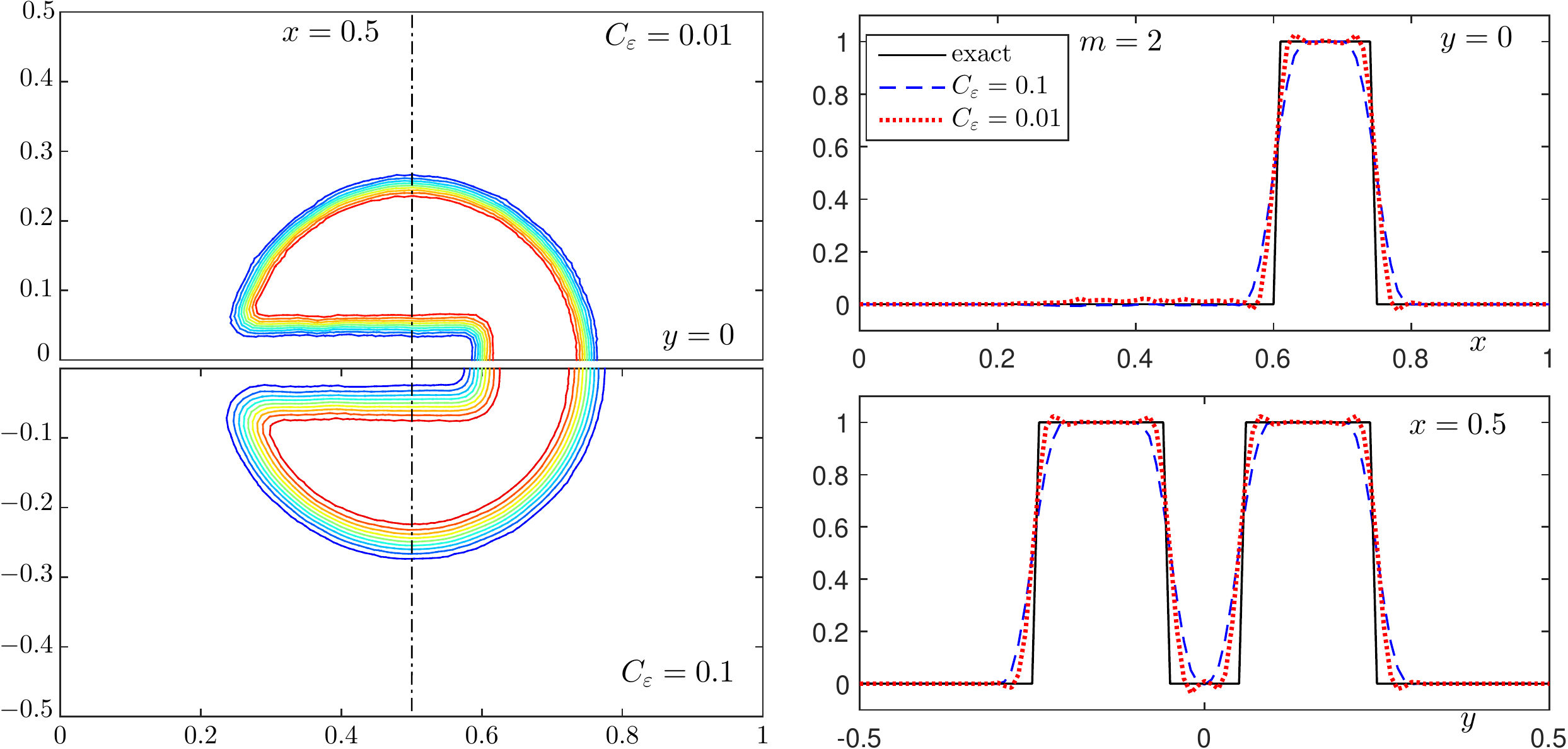}}
\end{center}
\caption{Slotted cylinder after one revolution for quadratic finite elements
$m=2$. Upper panel: three dimensional view of the solutions. Lower panel:
the level lines (on the left) and cross sections (on the right) that
correspond with the figures of the upper panel.}
\label{figure:7}
\end{figure}


Finally, we represent in Figure \ref{figure:8} the time evolution of the
maximum and minimum of the numerical solutions obtained by the conventional
LG method, and the DC-LG one (with $C_{\varepsilon }=0.01$ and $%
C_{\varepsilon }=0.1$). As we commented above, our calculation of the
numerical initial condition allows the generation of wiggles at the first
time step, in fact, the largest amplitude of such wiggles is $0.3$. The
DC-LG method with $C_{\varepsilon }=0.1$ dissipates these wiggles as the
solution progresses, such that the for $m=2$ the dissipation is very strong
at the beginning, going very quickly the minimum to zero and the maximum to
1, as, on the other hand, should be; however, when $C_{\varepsilon }=0.01$,
the wiggles are also dissipated, but at a slower rate, with the amplitudes
of the minimum and maximum values decreasing somewhat oscillatorily, tending
to $-0.05$ and $1.05$ respectively. However, though we proof that linear
polynomials are stable in the maximum norm, the behavior of the maximum and
minimum is not as good as that of quadratic elements; for instance, when $%
C_{\varepsilon }=0.1$ the dissipation of the amplitude of the wiggles is
slower and less strong than in the case of quadratic elements, noting that
the steady maximum and minimum are $1.03$ and $-0.03$ respectively; when $%
C_{\varepsilon }=0.01$ the maximum and minimum of DC-LG solution, though
smaller in amplitude, exhibit a similar oscillatory behavior as those of the
conventional LG method. It is remarkable that both the maximum and the
minimum of the conventional LG method, either with $m=1$ or $m=2$, undergo
dissipation at the beginning of the calculations and then go on exhibiting
an oscillatory behavior.

\begin{figure}[th!]
\begin{center}
\scalebox{0.46}{\includegraphics{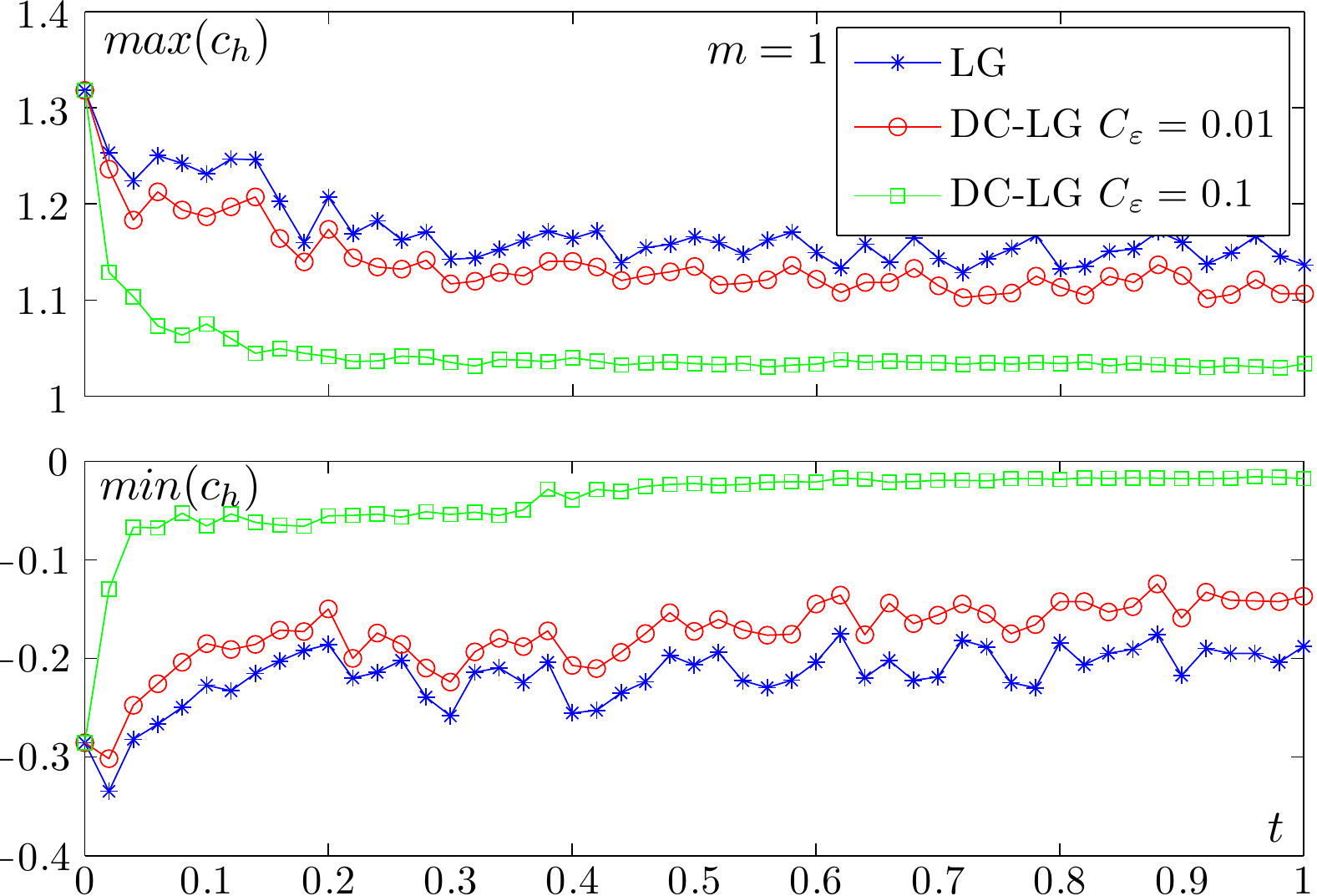}}\ \ \ %
\scalebox{0.46}{\includegraphics{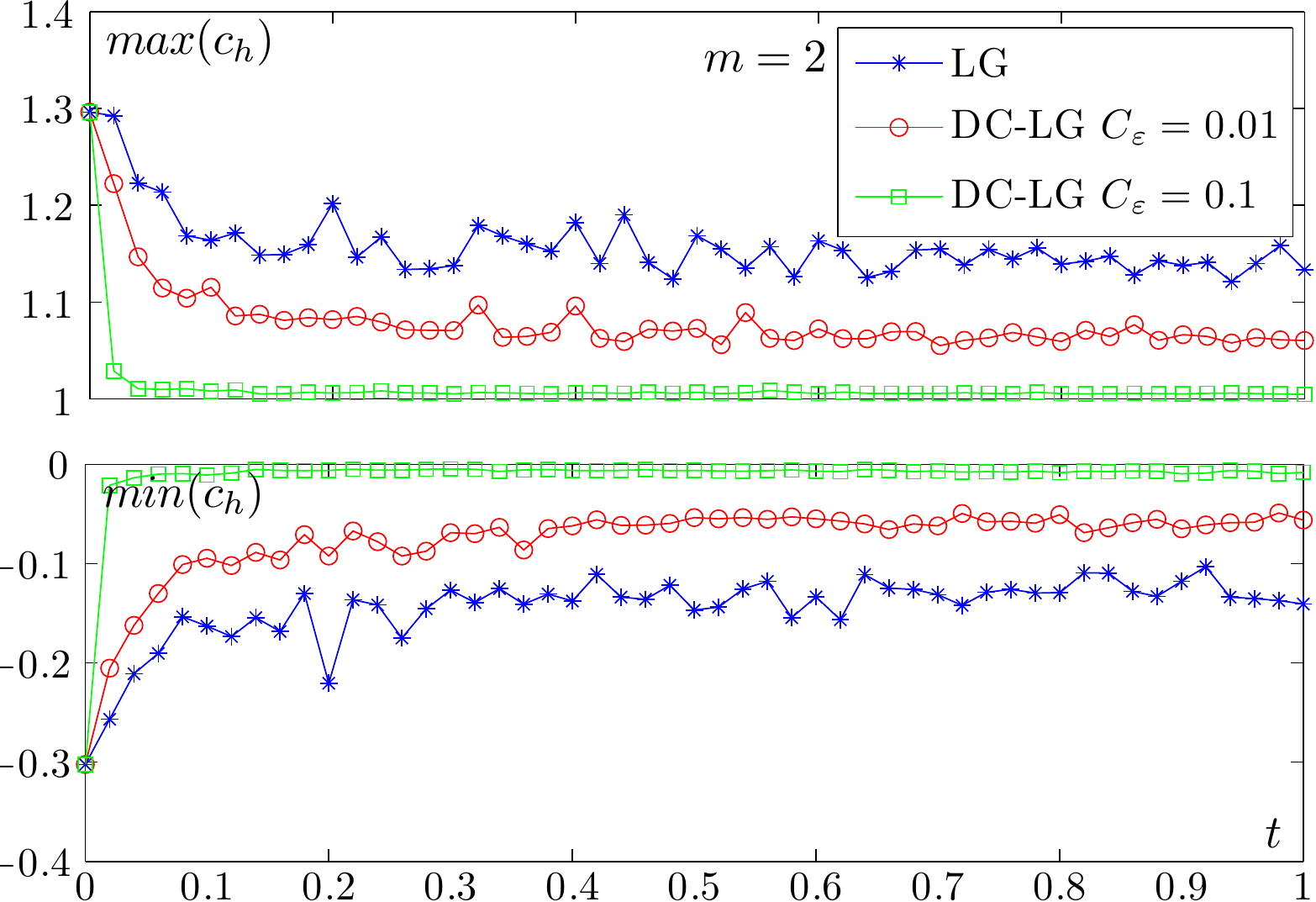}}
\end{center}
\caption{Evolution with time of the maximum and minimum of the slotted
cylinder during one revolution for linear $m=1$ and quadratic $m=2$ finite
elements}
\label{figure:8}
\end{figure}

\section{Concluding remarks}

1) We have obtained a new error estimate of the conventional LG method for
the advection equation. In contrast with previous estimates, ours is valid
for all $\Delta t$, no matter how small $\Delta t$ is, showing that for $%
\Delta t\leq Kh^{p},\ p>2$, the error is $O(h^{m})$, and for $\Delta
t>Kh^{p} $ the error is $O({h^{m+1}}/{\Delta t^{1/2}})$, here $K=(\left\Vert
\mathbf{u}\right\Vert _{L^{\infty }(L^{\infty }(D)^{d})})^{-1/2}$. This
error estimate has been obtained under the assumption that the integrals $%
\int_{K}\phi _{j}(X_{h}(x,t_{n+1},t_{n}))\phi _{i}(x)dx$ are calculated
exactly. 2) To validate our theoretical result we perform numerical tests
using quadrature rules of different orders to evaluate those integrals and
calculating exactly the trajectories. We find that the higher the order of
the quadrature rule the closer the error behavior to the theoretical one.
Other interesting finding is that for $\Delta t=O(h)$ and $\Delta t=O(h^{3})$%
, the error is quite independent of the order of the quadrature rule as long
as the rule calculates exactly polynomials of degree $\geq 2(m+1)$. 3) The
LG approach is a natural way of introducing upwinding in the numerical
method, but the degree of upwinding is not strong enough if the initial
condition lacks regularity. One way of stabilizing the conventional LG
method is using the so called local projection stabilization technique,
which is symmetric and acts on the small unresolved scales. We thus obtain
the so called LPS-LG method and estimate its error in a mesh dependent norm.
4) Neither the LPS-LG nor the conventional LG methods are stable in the
maximum norm, so they do not deal satisfactorily with strongly discontinuous
initial conditions. Following the idea of shock-capturing characteristic
streamline-diffusion method of \cite{john}, we have formulated the DC-LG
method that is a residual stabilized LG method, which for linear finite
elements is stable in both the $L^{2}$- and $L^{\infty }$-norms. This method
has shown to be effective in preserving the shape of the initial condition,
in particular, when quadratic elements are used, though there is no
theoretical proof of the stability in the infinite norm for these elements.
Finally, we must say that this dependence of the error behavior on the CFL
number of the LG methods is not exclusive for the pure advection problem, it
can also be proven for advection-dominated and NS problems, see \cite{BS1},
\cite{BS2} and \cite{BS3}.



\section*{Acknowledgements}

This research has been partially funded by grant PGC-2018-097565-B100 of
Ministerio de Ciencia, Innovaci\'{o}n y Universidades of Spain and of the
European Regional Development Fund.

\end{document}